\documentclass[11pt]{article}
\usepackage{amsmath, amssymb,verbatim,appendix}
\usepackage[mathscr]{eucal}
\usepackage{amscd}
\usepackage{amsthm}
\usepackage{enumerate}
\usepackage{cite}
\usepackage[margin=1in]{geometry}

\newtheorem{theorem}{Theorem}
\newtheorem{lemma}{Lemma}
\newtheorem{corollary}{Corollary}
\newtheorem{proposition}{Proposition}

\newtheorem{claim}{Claim}
\newtheorem{fact}{Fact}

\theoremstyle{definition}
\newtheorem{definition}{Definition}

\newcommand{\calH}{\mathcal{H}}

\newcommand{\calP}{\mathcal{P}}

\newcommand{\calG}{\mathcal{G}}

\newcommand{\calS}{\mathcal{S}}

\newcommand{\ex}{{\rm{ex}}}
\newcommand{\exs}{{\rm{ex}}_{\Sigma}}
\newcommand{\expi}{{\rm{ex}}_{\Pi}}

\title{
Extremal theory of locally sparse  multigraphs
}

\author{
Dhruv Mubayi \footnote{Department of Mathematics, Statistics, and Computer Science, University of Illinois at 
Chicago. Research supported in part by NSF Grant DMS 1300138; mubayi@uic.edu}
\and
Caroline Terry \footnote{Department of Mathematics, 
University of Maryland, College Park; 
cterry@umd.edu}
}

\begin{document}

\maketitle

\begin{abstract} 
An $(n,s,q)$-graph is an $n$-vertex multigraph where every set of $s$ vertices spans at most $q$ edges. In this paper, we determine the maximum product of the edge multiplicities in $(n,s,q)$-graphs if the congruence class of $q$ modulo ${s\choose 2}$ is in a certain interval of length about $3s/2$.  The smallest case that falls outside this range is $(s,q)=(4,15)$, and here the answer is  $a^{n^2+o(n^2)}$ where $a$ is transcendental assuming Schanuel's conjecture.  This could indicate the difficulty of solving the problem in full generality.   Many of our results can be seen as extending work by Bondy-Tuza~\cite{bondytuza} and F\"uredi-K\"undgen~\cite{FK} about sums of edge multiplicities to the product setting.  

We also prove a variety of other extremal results for $(n,s,q)$-graphs, including product-stability theorems.  These results are of additional interest because they can be used to enumerate and to prove logical 0-1 laws for $(n,s,q)$-graphs. Our work therefore extends many classical enumerative results in extremal graph theory beginning with the Erd\H os-Kleitman-Rothschild theorem \cite{EKR} to multigraphs.

\end{abstract}

\section{Introduction}


Given a set $X$ and a positive integer $t$, let ${X\choose t}=\{Y\subseteq X: |Y|=t\}$.  A \emph{multigraph} is a pair $(V,w)$, where $V$ is a set of vertices and $w:{V\choose 2}\rightarrow \mathbb{N}=\{0,1,2, \ldots\}$. 

\begin{definition} Given integers $s\geq 2$ and $q\geq 0$, a multigraph $(V,w)$ is an \emph{$(s,q)$-graph} if for every $X\in {V\choose s}$ we have $\sum_{xy\in {X\choose 2}}w(xy)\leq q$.  An $(n,s,q)$-graph is an $(s,q)$-graph with $n$ vertices, and $F(n,s,q)$ is the set of $(n,s,q)$-graphs with vertex set $[n]:=\{1,\ldots, n\}$. \end{definition}

 The goal of this paper is to investigate extremal, structural, and enumeration problems for $(n,s,q)$-graphs for a large class of pairs 
 $(s,q)$.

\begin{definition}
Given a multigraph $G=(V,w)$, define 
$$
S(G)=\sum_{xy\in {V\choose 2}} w(xy)\qquad \hbox{ and } \qquad P(G)=\prod_{xy\in {V\choose 2}}w(xy),
$$
$$
\exs(n,s,q) = \max \{S(G): G\in F(n,s,q)\}\quad \hbox{ and }\quad \expi(n,s,q)=\max \{P(G): G\in F(n,s,q)\}.
$$
An $(n,s,q)$-graph $G$ is \emph{sum-extremal} (\emph{product-extremal}) if $S(G)=\exs(n,s,q)$ ($P(G)=\expi(n,s,q)$). Let $\calS(n,s,q)$ ($\calP(n,s,q)$) be the set of all sum-extremal (product-extremal) $(n,s,q)$-graphs with vertex set $[n]$.  
\end{definition}
In \cite{bondytuza}, Bondy and Tuza determine the structure of multigraphs in $\calS(n,s,q)$ when $n$ is large compared to $s$ and $q\equiv 0, -1\pmod {{s\choose2}}$ and when $s=3$. In \cite{furedikundgen}, F\"{u}redi and K\"{u}ndgen (among other things) determine the asymptotic value of $\exs(n,s,q)$ for all $s, q$ with a $O(n)$ error term, and the exact value is determined for many cases.  Other special cases of these questions have appeared in \cite{kuchenbrod}.  A natural next step from the investigation of extremal problems for $(n,s,q)$-graphs is to consider questions of structure and enumeration. The question of enumeration for $(n,s,q)$-graphs was first addressed in \cite{MT415}, where it was shown the problem is closely related extremal results for the product of the edge multiplicities.

\begin{definition}
Given integers $s\geq 2$ and $q\geq {s\choose 2}$, define the \emph{asymptotic product density} and the \emph{asymptotic sum density}, respectively, as the following limits (which both exist): 
$$
\expi(s,q)=\lim_{n\rightarrow \infty} \Big(\expi(n,s,q)\Big)^{\frac{1}{{n\choose 2}}}\qquad \hbox{ and }\qquad \exs(s,q)=\lim_{n\rightarrow \infty} \frac{\exs(n,s,q)}{{n\choose 2}}.
$$
\end{definition}

In \cite{MT415}, the current authors showed $\expi(s,q)$ exists for all $s\geq 2$ and $q\geq 0$ and proved the following enumeration theorem for $(n,s,q)$-graphs in terms of $\expi(s,q+{s\choose 2})$.  

\begin{theorem}{\bf(~\cite{MT415})} \label{counting} Suppose $s\geq 2$ and $q\geq 0$ are integers. If $\expi(s, q+{s\choose 2})>1$, then  
$$
\expi\Big(s,q+{s\choose 2}\Big)^{{n\choose 2}}\leq |F(n,s,q)|\leq \expi\Big(s,q+{s\choose 2}\Big)^{(1+o(1)){n\choose 2}},
$$
and if $\expi(s,q+{s\choose 2})\leq 1$, then $|F(n,s,q)|\leq 2^{o(n^2)}$.
\end{theorem}

This result was used in \cite{MT415} along with a computation of $\expi(4,15)$ to give an enumeration of $F(n,4,9)$.  This case was of particular interest because it turned out that $|F(n,4,9)|=a^{n^2+o(n^2)}$, where $a$ is transcendental under the assumption of Schanuel's conjecture. In this paper, we continue this line of investigations by proving enumeration results for further cases of $s$ and $q$, and in some cases proving approximate structure theorems
 (the particular special case $(s,q)=(3,4)$ was recently studied in~\cite{FROSU}).   This generalizes many classical theorems about enumeration in extremal graph theory (beginning with the Erd\H os-Kleitman-Rothschild theorem \cite{EKR})
to the multigraph setting.
All of these results rely on computing $\expi(n,s,q)$, characterizing the elements in $\calP(n,s,q)$, and proving corresponding product-stability theorems, and this is the main content of this paper. Questions about $\expi(n,s,q)$ and $\calP(n,s,q)$ may also be of independent interest, as they are natural ``product versions'' of the questions about extremal sums for $(n,s,q)$-graphs investigated in \cite{bondytuza,furedikundgen}.  

\section{Main Results}
Given a multigraph $G=(V,w)$ and $xy\in {V\choose 2}$, we will refer to $w(xy)$ as the \emph{multiplicity} of $xy$.  The \emph{multiplicty of $G$} is $\mu(G)=\max\{w(xy): xy\in {V\choose 2}\}$. Our first main result, Theorem \ref{genst} below, gives us information about the asymptotic properties of elements in $F(n,s,q)$, in the case when $\expi(s,q+{s\choose 2})>1$.  \noindent Suppose $G=(V,w)$ and $G'=(V,w')$ are multigraphs.  We say that $G$ is a  \emph{submultigraph} of $G'$ if $V=V'$ and for each $xy\in {V\choose 2}$, $w(xy)\leq w'(xy)$.  Define $G^+=(V,w^+)$ where for each $xy\in {V\choose 2}$, $w^+(xy)=w(xy)+1$. Observe that if $G\in F(n,s,q)$, then $G^+\in F(n,s,q+{s\choose 2})$.  

\begin{definition}
Suppose $\epsilon>0$ and $n,s,q$ are integers satisfying $n\geq 1$, $s\geq 2$, and $q\geq 0$.  Set
\begin{align*}
\mathbb{E}(n,s,q,\epsilon)=\Big\{G\in F(n,s,q): P(G^+)>\expi\Big(s,q+{s\choose 2}\Big)^{(1-\epsilon){n\choose 2}}\Big\}.
\end{align*}
Then set $E(n,s,q,\epsilon)=\{G\in F(n,s,q): G\text{ is a submultigraph of some }G'\in \mathbb{E}(n,s,q,\epsilon)\}$.
\end{definition}

\begin{theorem}\label{genst}
Suppose $s\geq 2$ and $q\geq 0$ are integers satisfying $\expi(s,q+{s\choose 2})>1$.  Then for all $\epsilon>0$, there is $\beta>0$ such that for all sufficiently large $n$, the following holds.
\begin{align}\label{gensteq}
\frac{|F(n,s,q)\setminus E(n,s,q,\epsilon)|}{|F(n,s,q)|}\leq 2^{-\beta n^2}.
\end{align} 
\end{theorem}
Theorem \ref{genst} will be proved in Section \ref{aasection} using a consequence of a version of the hypergraph containers theorem for multigraphs from \cite{MT415}.  Our next results investigate $\expi(n,s,q)$ and $\calP(n,s,q)$ for various values of $(s,q)$.  Observe that if $q<{s\choose 2}$, then for any $n\geq s$, every $(n,s,q)$-graph $G$ must contain an edge of multiplicity $0$, and therefore $P(G)=0$.  Consequently, $\expi(n,s,q)=0$ and $\calP(n,s,q)=F(n,s,q)$, for all $n\geq s$.  For this reason we restrict our attention to the cases where $s\geq 2$ and $q\geq {s\choose 2}$.  Suppose $G=(V,w)$ and $G'=(V',w')$.  Then  $G=(V,w)$ and $G'=(V',w')$ are \emph{isomorphic}, denoted $G\cong G'$, if there is a bijection $f:V\rightarrow V'$ such that for all $xy\in {V\choose 2}$, $w(xy)=w'(f(x)f(y))$.  If $V=V'$, set $\Delta(G,G')=\big\{xy\in \binom{V}{2}: w(xy)\neq w'(xy)\big\}$.  Given $\delta>0$, $G$ and $G'$ are \emph{$\delta$-close} if $|\Delta(G,G')|\leq \delta n^2$, otherwise they are \emph{$\delta$-far}.    If $X\subseteq V$, $G[X]$ denotes the multigraph $(X,w\upharpoonright_{X\choose 2})$. Suppose that $q\equiv b\pmod{{s\choose 2}}$.  Our results fall into three cases depending on the value of $b$. 

\subsection{{\bf The case $0 \le b \le s-2$}}
\begin{definition}
Given $n\geq s\geq 1$ and $a\geq 1$, let $\mathbb{U}_{s,a}(n)$ be the set of multigraphs $G=([n],w)$ such that there is a partition $A_0, A_1,\ldots, A_{\lfloor \frac{n}{s}\rfloor}$ of $[n]$ for which the following holds.
\begin{itemize}
\item For each $1\leq i\leq \lfloor n/s\rfloor$, $|A_i|=s$, and $|A_0|=n-s\lfloor n/s\rfloor$.
\item For each $0\leq i\leq \lfloor n/s\rfloor$, and $G[A_i]$ is a star with $|A_i|-1$ edges of multiplicity $a+1$ and all other edges of multiplicity $a$.
\item For all $xy\notin \bigcup {A_i\choose 2}$, $w(xy)=a$.
\end{itemize}
Let $\mathbb{U}_a(n)$ be the unique element of $\mathbb{U}_{1,a}(n)$, i.e. $\mathbb{U}_a(n)=([n],w)$ where $w(xy)=a$ for all $xy\in {[n]\choose 2}$.
\end{definition}

\begin{theorem}\label{casei}
Suppose $n, s,q,a$ are integers satisfying $n\geq s\geq 2$, $a\geq 1$, and $q=a{s\choose 2}+b$ for some $0\leq b\leq s-2$.  
\begin{enumerate}[$\bullet$]
\item (Extremal) Then $a^{n\choose 2}\leq \expi(n,s,q)\leq a^{n\choose 2}((a+1)/a)^{\lfloor \frac{b}{b+1}n\rfloor}$ and thus $\expi(s,q)=a$.  Further, 
\begin{enumerate}
\item If $b=0$, then $\calP(n,s,q)=\{\mathbb{U}_{a}(n)\}$ and $\expi(n,s,q)=a^{n\choose 2}$.  
\item If $b=s-2$, then $\mathbb{U}_{s-1,a}(n)\subseteq \calP(n,s,q)$ and $\expi(n,s,q)=a^{n\choose 2}\Big(\frac{a+1}{a}\Big)^{\lfloor \frac{(s-2)n}{s-1}\rfloor}$.  Also, $\calP(n,3,q)=\mathbb{U}_{2,a}(n)$.
\end{enumerate}
\item (Stability) For all $\delta>0$, there is $\epsilon>0$ and $M$ such that for all $n>M$ and $G\in F(n,s,q)$, if $P(G)>\expi(n,s,q)^{1-\epsilon}$, then $G$ is $\delta$-close to $\mathbb{U}_a(n)$.
\end{enumerate}
\end{theorem}

One interesting phenomenon discovered in \cite{bondytuza} is that $\calS(n,3,3a+1)$ has many non-isomorphic multigraphs when $a\geq 1$ and $n$ is large. In contrast to this, Theorem \ref{casei} shows that all the multigraphs in $\calP(n,3,3a+1)=\mathbb{U}_{2,a}(n)$ are isomorphic.

\subsection{{\bf The case $b={s\choose 2}-t$ for some $1\leq t\leq \frac{s}{2}$}}
 Call a partition $U_1,\ldots, U_k$ of a finite set $X$ an \emph{equipartition} if $||U_i|-|U_j||\leq 1$ for all $i\neq j$. Recall the Tur\'{a}n graph, $T_s(n)$, is the complete $s$-partite graph with $n$ vertices, whose parts form an equipartition of its vertex set.
\begin{definition}
Given integers $a\geq 2$ and $n\geq s\geq 1$, define $\mathbb{T}_{s,a}(n)$ to be the set of multigraphs $G=([n],w)$ with the following property.  There is an equipartition $U_1,\ldots, U_{s}$ of $[n]$ such that
\[
w(xy)=\begin{cases} a-1 & \text{ if }xy\in {U_i\choose 2} \text{ for some }i\in [s].\\
a & \text{ if }(x,y)\in U_i \times U_j \text{ for some }i\neq j\in [s].
\end{cases}
\]
\end{definition}
\noindent We think of elements of $\mathbb{T}_{s,a}(n)$ as multigraph analogues of Tur\'{a}n graphs.  Let $t_{s}(n)$ be the number of edges in $T_s(n)$.

\begin{theorem}\label{caseii}
Let $s,q,a, t$ be integers satisfying $a\geq 2$, $q=a{s\choose 2}-t$ and either
\begin{enumerate}[(a)]
\item $s\geq 2$ and $t=1$ or 
\item $s\geq 4$ and $2\leq t\leq \frac{s}{2}$.
\end{enumerate}
\begin{enumerate}[$\bullet$]
\item (Extremal) Then for all $n\geq s$, $\mathbb{T}_{s-t,a}(n)\subseteq \calP(n,s,q)$, $\expi(n,s,q)=(a-1)^{n\choose 2}(\frac{a}{a-1})^{t_{s-t}(n)}$, and $\expi(s,q)=(a-1)(\frac{a}{a-1})^{\frac{s-t-1}{s-t}}$.  If (a) holds and $n\geq s$ or (b) holds and $n$ is sufficiently large, then $\calP(n,s,q)=\mathbb{T}_{s-t,a}(n)$.
\item (Stability) For all $\delta>0$, there is $M$ and $\epsilon$ such that for all $n>M$ and $G\in F(n,s,q)$, if $P(G)>\expi(n,s,q)^{1-\epsilon}$ then $G$ is $\delta$-close to an element of $\mathbb{T}_{s-t,a}(n)$.
\end{enumerate}
\end{theorem}

\subsection{{\bf The case  $(s,q)=(4,9)$}}
 The case $(s,q)=(4,9)$ is the first pair where $s\geq 2$ and $q\geq {s\choose 2}$ which is not covered by Theorems \ref{casei} and \ref{caseii}, and is closely related to an old question in extremal combinatorics. Let $\ex(n,\{C_3,C_4\})$ denote the maximum number of edges in a graph on $n$ vertices which contains no $C_3$ or $C_4$ as a non-induced subgraph.

\begin{theorem}\label{caseiv1}
 $\expi(n,4,9)=2^{\ex(n,\{C_3,C_4\})}$ for all $n\geq 4$.
\end{theorem}

\noindent
It is known that
$$\left(\frac{1}{2\sqrt 2}+o(1)\right) n^{3/2} < \ex(n,\{C_3,C_4\})
<\left(\frac{1}{2}+o(1)\right) n^{3/2}$$
and an old conjecture of Erd\H{o}s and Simonovits \cite{Erdos1} states that the lower bound is correct.  

The next case not covered here is 
$(s,q)=(4,15)$ and it was shown in~\cite{MT415} that $\expi(n, 4, 15) = 2^{\gamma n^2+o(n^2)}$ where $\gamma$ is transcendental and $2^{\gamma}$ is also transcendental if we assume Schanuel's conjecture from number theory. Many other cases were conjectured in~\cite{MT415}
to have transcendental behaviour like the case $(4,15)$.  This suggests that determining $\expi(s,q)$ for all pairs $(s,q)$ will be a hard problem. 

\subsection{Enumeration and structure of most $(n,s,q)$-graphs}

 Combining the extremal results of Theorems \ref{casei}, \ref{caseii}, and \ref{caseiv1} with Theorem \ref{counting} we obtain Theorem \ref{mgcor*} below, which enumerates $F(n,s,q)$ for many cases of $(s,q)$. 

\begin{theorem}\label{mgcor*}
Let $s,q,a,b$ be integers satisfying $s\geq 2$, $a\geq 0$, and $q=a{s\choose 2}+b$. 
\begin{enumerate}[(i)]
\item If $0\leq b\leq s-2$, then $|F(n,s,q)|=(a+1)^{{n\choose 2}}2^{o(n^2)}$.
\item If $b={s\choose 2}-t$ where $2\leq t\leq \frac{s}{2}$, then $|F(n,s,q)|=(a+1)^{n\choose 2}(\frac{a+2}{a+1})^{t_{s-t}(n)+o(n^2)}$.
\item $|F(n,4,3)|=2^{\Theta(n^{3/2})}$.
\end{enumerate}
\end{theorem}

In our last main result, Theorem \ref{mgcount2} below, we combine the stability results of Theorems \ref{casei} and \ref{caseii} with Theorem \ref{genst} to prove approximate structure theorems for many $(s,q)$.    Given $\delta>0$ and a set $E(n)\subseteq F(n,s,q)$, let $E^{\delta}(n)$ be the set of $G\in F(n,s,q)$ such that $G$ is $\delta$-close to some $G'\in E(n)$.

\begin{definition}\label{UT}
Suppose $n,a,s$ are integers such that $n,s\geq 1$.
\begin{enumerate}[(i)]
\item If $a\geq 1$, set $U_{a}(n)=\{G=([n],w): G\text{ is a submultigraph of some }G'\in \mathbb{U}_{a}(n)\}$.  
\item If $a\geq 2$, set $T_{s,a}(n)=\{G=([n],w): G\text{ is a submultigraph of some }G'\in \mathbb{T}_{s,a}(n)\}$.
\end{enumerate}
Observe that in each case, $\mathbb{U}_a(n)\subseteq U_a(n)$ and $\mathbb{T}_{s,a}(n)\subseteq T_{s,a}(n)$.
\end{definition}

\begin{theorem}\label{mgcount2}

Suppose $s,q,a,t,b$ are integers such that $n\geq s\geq 2$, and $E(n)$ is a set of multigraphs such that one of the following holds.
\begin{enumerate}[(i)]
\item $a\geq 0$, $q=a{s\choose 2}+b$ for some $0\leq b\leq s-2$, and $E(n)=U_a(n)$.  
\item $a\geq 1$, $q=a{s\choose 2}-t$ for some $1\leq t\leq \frac{s}{2}$, and $E(n)=T_{s-t,a}(n)$. 
\end{enumerate}
Then for all $\delta>0$ there exists $\beta>0$ such that for all sufficiently large $n$, 
\begin{align}\label{i}
\frac{|F(n,s,q)\setminus E^{\delta}(n)|}{|F(n,s,q)|}\leq 2^{-\beta {n\choose 2}}.
\end{align}
\end{theorem}

\section{Proof of Theorems~\ref{mgcor*} and \ref{mgcount2}}
In this section we prove Theorems~\ref{mgcor*} and \ref{mgcount2} assuming Theorems~\ref{genst}, \ref{casei}, and \ref{caseii}.
\bigskip

\noindent
{\bf Proof of Theorem~\ref{mgcor*}.}
Suppose first that case (i) holds.  By Theorem \ref{casei} (Extremal), 
$$
\expi\Big(s,q+{s\choose 2}\Big)=\expi\Big(s, (a+1){s\choose 2}+b\Big)=a+1.
$$
If $a=0$, then $\expi(s,q+{s\choose 2})=1$, so Theorem \ref{counting} implies $|F(n,s,q)|=2^{o(n^2)}=(a+1)^{n\choose 2}2^{o(n^2)}$.  If $a\geq 1$, then $\expi(s,q+{s\choose 2})=a+1>1$, so Theorem \ref{counting} implies 
$$
|F(n,s,q)|=(a+1)^{{n\choose 2}+o(n^2)}=(a+1)^{n\choose 2}2^{o(n^2)}.
$$
Suppose now that case (ii) holds.  So $q=a{s\choose 2}+{s\choose 2}-t=(a+1){s\choose 2}-t$.  By Theorem \ref{caseii} (Extremal), 
$$
\expi\Big(s,q+{s\choose 2}\Big)=\expi\Big(s, (a+2){s\choose 2}-t\Big)=(a+1)\Big(\frac{a+2}{a+1}\Big)^{\frac{s-t-1}{s-t}}.
$$
Since $a\geq 0$, this shows $\expi(s,q+{s\choose 2})>1$, so Theorem \ref{counting} implies 
$$
|F(n,s,q)|=\Big((a+1)\Big(\frac{a+2}{a+1}\Big)^{\frac{s-t-1}{s-t}}\Big)^{{n\choose 2}+o(n^2)}=(a+1)^{n\choose 2}\Big(\frac{a+2}{a+1}\Big)^{t_{s-t}(n)+o(n^2)}.
$$
For (iii) first observe that any subgraph of a graph of girth at least 5 is a $(4,3)$-graph, and since $\ex(n, \{C_3,C_4\})\geq c_1n^{3/2}$ for some constant $c_1>0$ (see \cite{Erdos1}) we obtain the lower bound. For the upper bound, observe that in a $(4,3)$-graph, there is at most one pair with multiplicity at least two and the set of pairs with multiplicity one forms a graph with no $C_4$. By the Kleitman-Winston theorem~\cite{KW}, the number of ways to choose the pairs of multiplicity one is at most $2^{c_2 n^{3/2}}$ for some constant $c_2>0$ and this gives the upper bound. 
\qed

\bigskip

\noindent
{\bf Proof of Theorem~\ref{mgcount2}.}
Fix $\delta>0$.  Observe that if case (i) holds (respectively, case (ii)), then $(s,q+{s\choose 2})$ satisfies the hypotheses of Theorem \ref{casei} (respectively, Theorem \ref{caseii}).  Let 
\[
\mathbb{E}(n)=\begin{cases} \mathbb{U}_{a+1}(n) &\text{ in case (i)}\\
\mathbb{T}_{s-t,a+1}(n)&\text{ in case (ii)}\end{cases}\]
By Theorem \ref{casei} (Stability) in case (i) and Theorem \ref{caseii} (Stability) in case(ii), there is $\epsilon>0$ so that for sufficiently large $n$, if $G^+\in F(n,s,q+{s\choose 2})$ satisfies $P(G^+)>\expi(n,s,q+{s\choose 2})^{1-\epsilon}$, then $G^+$ is $\delta$-close to some $G'\in \mathbb{E}(n)$.  Note that $G'\in \mathbb{E}(n)$ implies there is $H\in E(n)$ such that $H^+=G$.  Combining this our choice of $\epsilon$, we obtain the following.  For all sufficienlty large $n$ and $G\in F(n,s,q)$,
\begin{align}\label{caseist}
\text{if $P(G^+)>\expi\Big(n,s,q+{s\choose 2}\Big)^{1-\epsilon}$, then $G^+$ is $\delta$-close to $H^+$, for some $H\in E(n)$}.
\end{align}
By Theorem \ref{casei} (Extremal) in case (i) and Theorem \ref{caseii} (Extremal) in case(ii), we must have that $\expi(s,q+{s\choose 2})>1$.  So Theorem \ref{genst} implies there is $\beta>0$ such that for all sufficiently large $n$ the following holds.
\begin{align*}
\frac{|F(n,s,q)\setminus E(n,s,q,\epsilon)|}{|F(n,s,q)|}\leq 2^{-\beta n^2}.
\end{align*}
So to show (\ref{i}), it suffices to show that for sufficiently large $n$, $E(n,s,q,\epsilon)\subseteq E^{\delta}(n)$.  Fix $n$ sufficiently large and suppose $G=([n],w^G)\in E(n,s,q,\epsilon)$.  By definition, this means there is $G'\in F(n,s,q)$ such that $P(G'^+)>\expi(n,s,q+{s\choose 2})^{1-\epsilon}$ and $G$ is a submultigraph of $G'$.  By (\ref{caseist}), $G'^+$ is $\delta$-close to $H^+$, for some $H\in E(n)$.  Define $H'=([n],w^{H'})$ such that $w^{H'}(xy)=w^G(xy)$ if $xy\in {[n]\choose 2}\setminus \Delta(G', H)$, and $w^{H'}(xy)=0$ if $xy\in \Delta(G',H)$.  We claim $H'$ is a submultigraph of $H$.  Fix $xy\in {[n]\choose 2}$.  We want to show $w^{H'}(xy)\leq w^H(xy)$.  If $xy\in \Delta(G',H)$, then $w^{H'}(xy)=0\leq w^H(xy)$ is immediate.  If $xy\notin \Delta(G',H)$, then $w^{H'}(xy)=w^G(xy)\leq w^{G'}(xy)=w^H(xy)$, where the inequality is because $G$ is a submultigraph of $G'$ and the last equality is because $xy\notin \Delta(G',H)$.  Thus $H'$ is a submultigraph of $H\in E(n)$, which implies $H'$ is also in $E(n)$.  By definition of $H'$, $\Delta(G,H')\subseteq \Delta(G',H)=\Delta(G'^+,H^+)$.  Since $G'^+$ and $H^+$ are $\delta$-close, this implies $|\Delta(G,H')|\leq \delta n^2$, and $G\in E^{\delta}(n)$. 
\qed

\section{Proof of Theorem \ref{genst}}\label{aasection}

In this section we prove Theorem \ref{genst}.  We will use Theorem \ref{containerscor} below, which is a version of the hypergraph containers theorem of \cite{Baloghetal1, saxton-thomason} for multigraphs.  Theorem \ref{containerscor} was proved in \cite{MT415}.

\begin{definition}
Suppose $s\geq 2$ and $q\geq 0$ are integers.  Set 
$$
\calH(s,q)=\{G=([s],w):\mu(G)\leq q\text{ and }S(G)>q\},\quad \text{ and }\qquad g(s,q)=|\calH(s,q)|.
$$
If $G=(V,w)$ is a multigraph, let $\calH(G,s,q)=\{X\in {V\choose s}: G[X]\cong G'\text{ for some }G'\in \calH(s,q)\}$. 
\end{definition}

\begin{theorem}\label{containerscor}
For every $0<\delta<1$ and integers $s\geq 2$, $q\geq 0$, there is a constant $c=c(s,q,\delta)>0$ such that the following holds.  For all sufficiently large $n$, there is $\mathcal{G}$ a collection of multigraphs of multiplicity at most $q$ and with vertex set $[n]$ such that
\begin{enumerate}[(i)]
\item for every $J\in F(n,s,q)$, there is $G\in \mathcal{G}$ such that $J$ is a submultigraph of $G$,
\item for every $G\in \mathcal{G}$, $|\calH(G,s,q)|\leq \delta {n\choose s}$, and 
\item $\log |\mathcal{G}| \leq cn^{2-\frac{1}{4s}}\log n$.
\end{enumerate}
\end{theorem}

We will also use the following two results appearing in \cite{MT415}.  

\begin{lemma}[Lemma 1 of \cite{MT415}]\label{triangleremoval1}
Fix integers $s\geq 2$ and $q\geq 0$.  For all $0<\nu<1$, there is $0<\delta<1$ such that for all sufficiently large $n$, the following holds.  If $G=([n],w)$ satisfies $\mu(G)\leq q$ and $|\calH(G,s,q)|\leq \delta{n\choose 2}$, then $G$ is $\nu$-close to some $G'$ in $F(n,s,q)$.
\end{lemma}

\begin{proposition}[Proposition 2 in \cite{MT415}]\label{density}
For all $n\geq s\geq 2$ and $q\geq 0$, $\expi(s,q)$ exists and $\expi(n,s,q)\geq \expi(s,q)^{n\choose 2}$.  If $q\geq {s\choose 2}$, then $\expi(s,q)\geq 1$.
\end{proposition}

\noindent {\bf Proof of Theorem \ref{genst}.}
Fix $\epsilon>0$ and set $\nu=(\epsilon \log(\expi(s,q+{s\choose 2}))/(8\log (q+1))$.  Choose $\delta>0$ according to Lemma \ref{triangleremoval1} so that the following holds for all sufficiently large $n$. 
\begin{align}\label{tm}
\text{ Any $G=([n],w)$ with $\mu(G)\leq q$ and $|\calH(G,s,q)|\leq \delta {n\choose 2}$ is $\nu$-close to some $G'$ in $F(n,s,q)$}.
\end{align}
Fix $n$ sufficiently large.  Apply Theorem \ref{containerscor} to obtain a constant $c$ and a collection $\mathcal{G}$ of multigraphs of multiplicity at most $q$ and with vertex set $[n]$ satisfying (i)-(iii) of Theorem \ref{containerscor}.  Suppose that $H=([n],w^H)\in F(n,s,q)\setminus E(n,s,q,\epsilon)$.  By (i), there is $G=([n],w^G)\in \calG$ such that $H$ is a submultigraph of $G$ and $|\calH(G,s,q)|\leq \delta {n\choose s}$.  We claim that $P(G^+)\leq \expi(n,s,q+{s\choose 2})^{1-\epsilon/2}$.  Suppose towards a contradiction this is not the case, so $P(G^+)> \expi(n,s,q+{s\choose 2})^{1-\epsilon/2}$.  By (\ref{tm}), $|\calH(G,s,q)|\leq \delta {n\choose s}$ implies there is $G'=([n],w^{G'})\in F(n,s,q)$ which is $\nu$-close to $G$.  Define $H'=([n],w^{H'})$ by setting $w^{H'}(xy)=w^H(xy)$ for all $xy\in {[n]\choose 2}\setminus \Delta(G,G')$ and $w^{H'}(xy)=0$ for all $xy\in \Delta(G,G')$.  By construction and because $H'$ is a submultigraph of $G'$, we have that $H$ is also a submultigraph of $G'$.  Observe
$$
P(G'^+)=P(G^+)\Big(\prod_{xy\in \Delta(H,H')}\frac{w^{G'}(xy)+1}{w^G(xy)+1}\Big)\geq P(G^+)(q+1)^{-|\Delta(G,G')|},
$$
where the inequality is because $1\leq w^{G'}(xy)+1, w^G(xy)+1\leq q+1$ implies $\frac{w^{G'}(xy)+1}{w^G(xy)+1}\geq \frac{1}{q+1}$.  Combining this with the fact that $G$ and $G'$ are $\nu$-close, the definition of $\nu$, and our assumption that $P(G^+)\geq \expi(n,s,q+{s\choose 2})^{1-\epsilon/2}$, we have that $P(G'^+)$ is at least the following.
$$
P(G^+)(q+1)^{-\nu n^2}=P(G^+)\expi\Big(s,q+{s\choose 2}\Big)^{-\epsilon n^2/8}\geq \expi\Big(n,s,q+{s\choose 2}\Big)^{1-\epsilon/2}\expi\Big(s,q+{s\choose 2}\Big)^{-\epsilon n^2/8}.
$$
Since $\expi(n,s,q+{s\choose 2})^{1/{n\choose 2}}\geq \expi(s,q+{s\choose 2})$ (see Proposition \ref{density}), we obtain that  the right hand side is at least 
$$
\expi\Big(n,s,q+{s\choose 2}\Big)^{1-\epsilon/2}\expi\Big(n,s,q+{s\choose 2}\Big)^{-\epsilon n^2/(8{n\choose 2})}\geq \expi\Big(n,s,q+{s\choose 2}\Big)^{1-\epsilon},
$$
where the inequality is because $n$ large implies $\epsilon n^2/(8{n\choose 2})\leq \epsilon/2$.  But now $H$ is a submultigraph of $G'$ and $P(G'^+)\geq \expi(n,s,q+{s\choose 2})^{1-\epsilon}$, contradicting that $H\in F(n,s,q)\setminus E(n,s,q,\epsilon)$.  Therefore, every element of $F(n,s,q)\setminus E(n,s,q,\epsilon)$ can be constructed as follows.
\begin{enumerate}[$\bullet$]
\item Choose some $G\in \calG$ with $P(G^+)\leq \expi(n,s,q+{s\choose 2})^{1-\epsilon/2}$.  There are at most $cn^{2-\frac{1}{4s}}\log n$ choices.  Since $n$ is large and $\expi(s,q+{s\choose 2})>1$, we may assume $cn^{2-\frac{1}{4s}}\log n\leq \expi(s,q+{s\choose 2})^{\epsilon {n\choose 2}/4}$.
\item Choose a submultigraph of $G$.  There are at most $P(G^+)\leq \expi(n,s,q+{s\choose 2})^{1-\epsilon/2}$ choices.
\end{enumerate}
This shows
\begin{align*}
|F(n,s,q)\setminus E(n,s,q,\epsilon)|&\leq \expi\Big(s,q+{s\choose 2}\Big)^{\epsilon {n\choose 2}/4} \expi\Big(n,s,q+{s\choose 2}\Big)^{1-\epsilon/2}\\
&\leq \expi\Big(s,q+{s\choose 2}\Big)^{-\epsilon {n\choose 2}/4}\expi\Big(n,s,q+{s\choose 2}\Big),
\end{align*}
where the second inequality is because $\expi(n,s,q+{s\choose 2})\geq \expi(s,q+{s\choose 2})^{n\choose 2}$. By Theorem \ref{counting}, $|F(n,s,q)|\geq \expi(n,s,q)$, so this implies 
\begin{align*}
\frac{|F(n,s,q)\setminus E(n,s,q,\epsilon)|}{|F(n,s,q)|}&\leq \expi\Big(s,q+{s\choose 2}\Big)^{-\epsilon {n\choose 2}/4}.
\end{align*}
Setting $\beta=\frac{\epsilon}{4}\log_2(\expi(s,q+{s\choose 2})$ finishes the proof (note $\beta>0$ since $\expi(s,q+{s\choose 2})>1$).
\qed

\section{Extremal Results}

In this section we prove the extremal statements in Theorems \ref{casei} and \ref{caseii}.  We begin with some preliminaries.  Suppose $s\geq 2$ and $q\geq {s\choose 2}$.  It was shown in \cite{furedikundgen} that $\exs(s,q)$ exists, and the AM-GM inequality implies that 
\begin{align}\label{inequality}
\expi(s,q)=\lim_{n\rightarrow \infty} \expi(n,s,q)^{1/{n\choose 2}}\leq \lim_{n\rightarrow \infty} \frac{\exs(n,s,q)}{{n\choose 2}} = \exs(s,q).
\end{align}
The following lemma is an integer version of the AM-GM inequality.

\begin{lemma}\label{arithgeom}
If $\ell \geq 2$, $k\in [\ell]$ and $a, x_1,\ldots, x_{\ell}$ are positive integers such that $\sum_{i=1}^{\ell}x_i \leq a{\ell}-k$, then $\prod_{i=1}^{\ell}x_i\leq a^{\ell-k}(a-1)^k$.  Moreover, equality holds if and only if exactly $k$ of the $x_i$ are equal to $a-1$ and the rest are equal to $a$.
\end{lemma}

\begin{proof}
If there are $x_i$ and $x_j$ with $x_i<x_j-1$, then replacing $x_i$ with $x_i+1$ and replacing $x_j$ with $x_j-1$ increases the product and keeps the sum unchanged.  So no two of the $x_i$'s differ by more than one when the product is maximized.
\end{proof}

\begin{corollary}\label{sumimpliesproduct}
Let $n\geq s\geq 2$, $a\geq 2$, and $(a-1){s\choose 2}\leq q<a{s\choose 2}$.  Suppose $G\in \calS(n,s,q)$ has all edge multiplicities in $\{a,a-1\}$ and contains exactly $k$ edges of multiplicity $a-1$.  Then for all other $G'\in F(n,s,q)$, $G'\in \calP(n,s,q)$ if and only if $G'$ has  $k$ edges of multiplicity $a-1$ and all other edges of multiplicity $a$.  Consequently, $G\in \calP(n,s,q)\subseteq \calS(n,s,q)$.
\end{corollary}
\begin{proof}
Fix $G$ so that the hypotheses hold.  Then $S(G)=a{n\choose 2}-k$ and $P(G)=a^{{n\choose 2}-k}(a-1)^k$.  Let $G'=([n],w)$ be another element of $F(n,s,q)$.  Since $G\in \calS(n,s,q)$, we have
$$
S(G')\leq S(G)= a{n\choose 2}-k.
$$
By Lemma \ref{arithgeom} with $\ell={n\choose 2}$, $P(G')\leq a^{{n\choose 2}-k}(a-1)^k$ with equality if and only if $\{w(xy): xy\in {[n]\choose 2}\}$ consists of $k$ elements equal to $a-1$ and the rest equal to $a$.  This shows $G'\in \calP(n,s,q)$ if and only if $G'$ has $k$ edges of multiplicity $a-1$ and the rest of multiplicity $a$.  Consequently, $G\in \calP(n,s,q)$.  To show $\calP(n,s,q)\subseteq \calS(n,s,q)$, let $G'\in \calP(n,s,q)$.  Then by what we have shown, $S(G')=a{n\choose 2}-k=S(G)$, so $G\in \calS(n,s,q)$ implies $G'\in \calS(n,s,q)$.
\end{proof}

The following is a consequence of Theorem 5.2 in \cite{bondytuza} (case $b=0$) and Theorems 8 and 9 in \cite{furedikundgen} (cases $0<b\leq s-2$).
\begin{theorem}[{\bf Bondy-Tuza \cite{bondytuza}, F\"{u}redi-K\"{u}ndgen \cite{furedikundgen}}]\label{bleqs-2}
Let $n\geq s\geq 2$, $a\geq 1$, $0\leq b\leq s-2$, and $q=a{s\choose 2}+b$.  Then 
$$
\exs(n,s,q)\leq a{n\choose 2}+\Bigg\lfloor \frac{b}{b+1}n\Bigg\rfloor.
$$
with equality holding when $b=s-2$ and when $b=0$.
\end{theorem}

\noindent {\bf Proof of Theorem \ref{casei} (Extremal).}
Since $\mathbb{U}_a(n)\in F(n,s,q)$, $a^{n\choose 2}\leq \expi(n,s,q)$.  On the other hand, let $G\in F(n,s,q)$.  Theorem \ref{bleqs-2} implies that $S(G)\leq a{n\choose 2}+\lfloor \frac{b}{b+1}n\rfloor$.  This along with Lemma \ref{arithgeom} implies that $P(G)\leq a^{n\choose 2}((a+1)/a)^{\lfloor \frac{b}{b+1}n\rfloor}$. Thus $a^{n\choose 2}\leq \expi(n,s,q)\leq  a^{n\choose 2}((a+1)/a)^{\lfloor \frac{b}{b+1}n\rfloor}$, which implies $\expi(s,q)=a$.  

Case (a): If $b=0$, then Theorem \ref{bleqs-2} implies $\mathbb{U}_a(n)\in \calS(n,s,q)$.  Because $\mathbb{U}_a(n)$ has all edge multiplicities in $\{a\}$, Corollary \ref{sumimpliesproduct} implies $\mathbb{U}_a(n)\in \calP(n,s,q)$ and moreover, every other element of $\calP(n,s,q)$ has all edges of multiplicity $a$.  In other words, $\{\mathbb{U}_a(n)\}=\calP(n,s,q)$, so $\expi(n,s,q)=a^{n\choose 2}$.

Case (b): If $b=s-2$, then it is straightforward to check $\mathbb{U}_{s-1,a}(n)\subseteq F(n,s,q)$.  Since $S(G)=a{n\choose 2}+\lfloor \frac{s-2}{s-1}n\rfloor$ for all $G\in U_{s-1,a}(n)$, Theorem \ref{bleqs-2} implies $\mathbb{U}_{s-1,a}(n)\subseteq \calS(n,s,q)$.  Because every element in $\mathbb{U}_{s-1,a}(n)$ has all edge multiplicities in $\{a+1, a\}$, Corollary \ref{sumimpliesproduct} implies $\mathbb{U}_{s-1,a}(n)\subseteq \calP(n,s,q)$ and every $G'\in \calP(n,s,q)$ contains exactly $\lfloor \frac{s-2}{s-1}n\rfloor$ edges of multiplicity $a+1$, and all others of multiplicity $a$.  Thus $\expi(n,s,q)=a^{n\choose 2}(\frac{a+1}{a})^{\lfloor \frac{s-2}{s-1}\rfloor}$.  Suppose $s=3$, $b=1$, and $G'=([n],w)\in \calP(n,s,q)$.  If there are $x,y\neq z\in [n]$ such that $w(xy)=w(xz)=a+1$, then because $G'$ contains only edges of multiplicity $a+1$ and $a$, $S(\{x,y,z\})\geq 2(a+1)+a= 3a+2>q$, a contradiction.  Thus the edges of multiplicity $a+1$ form a matching of size $\lfloor \frac{n}{2}\rfloor$ in $G'$, so $G'\in \mathbb{U}_{s-1,a}(n)$.  This shows $\mathbb{U}_{s-1,a}(n)=\calP(n,s,q)$.
\qed

\vspace{3mm}
\noindent The following is a consequence of Theorem 5.2 of \cite{bondytuza}. 

\begin{theorem}[{\bf Bondy-Tuza \cite{bondytuza}}]\label{5.2}
Suppose $n\geq s\geq 2$, $a\geq 1$, and $q=a{s\choose 2}-1$.  Then 
$$
\exs(n,s,q)=(a-1){n\choose 2}+t_{s-1}(n).
$$
\end{theorem}

\noindent{\bf Proof of Theorem \ref{caseii}(a) (Extremal).}
Since $\mathbb{T}_{s-1,a}(n)\subseteq F(n,s,q)$ and for all $G\in \mathbb{T}_{s-1,a}(n)$, $S(G)=(a-1){n\choose 2}+t_{s-1}(n)$, Theorem \ref{5.2} implies that $\mathbb{T}_{s-1,a}(n)\subseteq \calS(n,s,q)$.  Therefore Corollary \ref{sumimpliesproduct} implies $\mathbb{T}_{s-1,a}(n)\subseteq \calP(n,s,q)$ and each $G\in \calP(n,s,q)$ has $t_{s-1}(n)$ edges of multiplicity $a$ and the rest of multiplicity $a-1$.  Fix $G=([n],w)\in \calP(n,s,q)$ and let $G'$ be the graph with vertex set $[n]$ and edge set $E=\{xy\in {[n]\choose 2}: w(xy)=a\}$.  Then $G'$ is $K_s$-free and has $t_{s-1}(n)$ edges, so by Tur\'{a}n's theorem, $G'=T_{s-1}(n)$ and thus $G\in \mathbb{T}_{s-1,a}(n)$.  So we have shown, $\calP(n,s,q)=\mathbb{T}_{s-1,a}(n)$.  Consequently, $\expi(n,s,q)=(a-1)^{n\choose 2}(\frac{a}{a-1})^{t_{s-1}(n)}$ and $\expi(s,q)=(a-1)(\frac{a}{a-1})^{\frac{s-2}{s-1}}$.
\qed

\vspace{3mm}

\noindent To prove Theorem \ref{caseii}(b) (Extremal), we will need the following theorem, as well as a few lemmas.
\begin{theorem}{\bf [Dirac \cite{Dirac}, Bondy-Tuza \cite{bondytuza}]}\label{schoose2-2}
Let $n\geq s\geq 4$, $a\geq 1$, and $q=a{s\choose 2}-t$ for some $2\leq t\leq \frac{s}{2}$.  Then $\exs(n,s,q)=\exs(n,s',q')$ where $s'=s-t+1$ and $q'=a{s'\choose 2}-1$.
\end{theorem}
\begin{proof}
Let $n\geq s\geq 4$ and  $2\leq t\leq s/2$.  In \cite{Dirac}, Dirac proved that $\exs(n,s,{s\choose 2}-t)=t_{s-t}(n)$.  This along with Lemma 5.1 in \cite{bondytuza} implies that for all $a\geq 1$, 
$$
\exs(n,s,a{s\choose 2}-t)=\exs(n,s,{s\choose 2}-t)+(a-1){n\choose 2}=t_{s-t}(n)+(a-1){n\choose 2}=\exs(n,s',a{s'\choose 2}-1),
$$
where the last equality is by Theorem \ref{5.2} applied to $s'$ and $a{s'\choose 2}-1$.
\end{proof}

\begin{lemma}\label{factii}
If $s,q,a,t$ are integers satisfying case (b) of Theorem \ref{caseii}, and $s'=s-t+1$,  $q'=a{s'\choose 2}-1$, then for all $n\geq s$, $\mathbb{T}_{s'-1}(n)\subseteq \calP(n,s,q)$ and $\expi(n,s,q)=\expi(n,s',q')$.
\end{lemma}
\begin{proof}
Set $s'=s-t+1$ and $q'=a{s'\choose 2}-1$, and fix $n\geq s$.  Fix $G\in \mathbb{T}_{s'-1,a}(n)$.  It is straightforward to check that $G\in F(n,s,q)$.  By Theorem \ref{schoose2-2}, $\exs(n,s',q')=\exs(n,s,q)$.  Since $S(G)=(a-1){n\choose 2}+t_{s'-1}(n)$, by Theorem \ref{5.2} applied to $s'$ and $q'$, we have that $S(G)=\exs(n,s',q')=\exs(n,s,q)$.  This shows $G\in \calS(n,s,q)$. By Corollary \ref{sumimpliesproduct}, since $G$ has all edge multiplicities in $\{a,a-1\}$, $G\in \calP(n,s,q)$, so $P(G)=\expi(n,s,q)$.  Since $G\in \mathbb{T}_{s'-1,a}(n)$ and $\mathbb{T}_{s'-1,a}(n)\subseteq \calP(n,s',q')$ by Theorem \ref{caseii}(a) (Extremal), $P(G)=\expi(n,s',q')$.  Thus $\expi(n,s,q)=P(G)=\expi(n,s',q')$.
\end{proof}
 
We now fix some notation.  Given $n\in \mathbb{N}$, $z\in [n]$, $Y\subseteq [n]$, and $G=([n],w)$, set 
$$
S(Y)=\sum_{xy\in {Y\choose 2}} w(xy),\quad S_z(Y)=\sum_{y\in Y}w(yz), \quad P(Y)=\prod_{xy\in {Y\choose 2}} w(xy), \quad \hbox{ and }\quad P_z(Y)=\prod_{y\in Y}w(yz)
$$
If $X\subseteq [n]$ is disjoint from $Y$, set $P(X,Y)=\prod_{x\in X, y\in Y}w(xy)$. 

\begin{claim}\label{gwthcaseii}
Suppose $s,q,a,t$ are integers satisfying the hypotheses of case (b) of Theorem \ref{caseii}.  Then for all $n\geq 2s$ and $s-t+1\leq y\leq s-1$, 
\begin{align*}
\expi(n-y,s,q)\leq \expi(n,s,q)(a-1)^{-{y\choose 2}}\Big(a^{y-2}(a-1)^2\Big)^{-(n-y)}\Big(\frac{a-1}{a}\Big)^{\frac{n-y}{s-t}}.
\end{align*}
\end{claim}
\begin{proof}
Set $s'=s-t+1$ and $q'=a{s'\choose 2}-1$. Fix $n\geq s$ and $s'\leq y\leq s-1$.  Choose some $H=([n-y],w)\in \mathbb{T}_{s'-1,a}(n-y)$ and let $U_1,\ldots, U_{s'-1}$ be the partition of $[n-y]$ corresponding to $H$.  Observe that there is some $i$ such that $|U_i|\geq \frac{n-y}{s'-1}$.  Without loss of generality, assume $|U_1|\geq \frac{n-y}{s'-1}$.  Assign the elements of $Y':=[n]\setminus [n-y]$ to the $U_i$ in as even a way as possible, to obtain an equipartition $U_1',\ldots, U_{s'-1}'$ of $[n]$ extending $U_1,\ldots, U_{s'-1}$.  Observe that because $s'\leq |Y'|\leq s-1$ and $s'-1=s-t\geq s/2$, for each $i$, $|U_i'\setminus U_i|\in \{1,2\}$, and there is at least one $i$ such that $|U_i'\setminus U_i|=1$.  Since $|U_1|\geq \frac{n-y}{s-t}$, by redistributing $Y'$ if necessary, we may assume that $|U_1'\setminus U_1|=1$.  Define a new multigraph $H'=([n],w')$ so that $w'(xy)=a-1$ if $xy\in {U'_i\choose 2}$ for some $i\in [s'-1]$ and $w'(xy)=a$ if $(x,y)\in U'_i\times U'_j$ for some $i\neq j$.  Note that by construction $H'\in \mathbb{T}_{s'-1,a}(n)$ and $H'[[n-y]]=H$.   By Lemma \ref{factii}, since $n-y\geq s$, $H\in \mathbb{T}_{s'-1,a}(n-y)$ and $H'\in \mathbb{T}_{s'-1,a}(n)$ imply $H\in \calP(n-y,s,q)$ and $H'\in \calP(n,s,q)$.  These facts imply the following.
\begin{align}\label{n}
\expi(n,s,q)=P(H')= P(H)P(Y')P(Y',[n-y])&= \expi(n-y,s,q)P(Y')P(Y',[n-y]).
\end{align}

By definition of $H'$, if $|U'_i\setminus U_i|=2$, then for all $z\in U_i$, $P_z(Y')=a^{y-2}(a-1)^2$ and if $|U_i'\setminus U_i|=1$, then for all $z\in U_i$, $P_z(Y')=a^{y-1}(a-1)$. Since $|U'_1\setminus U_1|=1$, this implies
\begin{align}\label{m}
P(Y',[n-y])\geq \Big(a^{y-2}(a-1)^2\Big)^{n-y-|U_{1}|}\Big(a^{y-1}(a-1)\Big)^{|U_1|}=\Big(a^{y-2}(a-1)^2\Big)^{n-y}\Big(\frac{a}{a-1}\Big)^{|U_1|}.
\end{align}
By construction, $P(Y')\geq (a-1)^{y\choose 2}$. Combining this with (\ref{n}), (\ref{m}), and the fact that $|U_1|\geq \frac{n-y}{s-t}$, we obtain
$$
\expi(n,s,q)\geq \expi(n-y,s,q)(a-1)^{y\choose 2}\Big(a^{y-2}(a-1)^2\Big)^{n-y}\Big(\frac{a}{a-1}\Big)^{\frac{n-y}{s-t}}.
$$
Rearranging this yields $\expi(n-y,s,q)\leq \expi(n,s,q)(a-1)^{-{y\choose 2}}(a^{y-2}(a-1)^2)^{-(n-y)}(\frac{a-1}{a})^{ \frac{n-y}{s-t}}$.  
\end{proof}

\begin{lemma}\label{stabilityiiilemma1}
Let  $n\geq s\geq 4$, $a\geq 2$, and $q=a{s\choose 2}-t$ for some $2\leq t\leq \frac{s}{2}$. Suppose $G\in F(n,s,q)$ and $Y\in {[n]\choose s-t+1}$ satisfies $S(Y)\geq a{s-t+1\choose 2}$.  Then there is $Y\subseteq Y'\subseteq [n]$ such that $s-t+1\leq |Y'|\leq s-1$ and for all $z\in [n]\setminus Y'$, $S_z(Y')\leq a|Y'|-2$, and consequently, $P_z(Y')\leq a^{|Y'|-2}(a-1)^2$.
\end{lemma}
\begin{proof}
Suppose towards a contradiction that $Y\in {[n]\choose s-t+1}$ satisfies $S(Y)\geq a{s-t+1\choose 2}$ but for all $Y\subseteq Y'\subseteq [n]$ such that $s-t+1\leq |Y'|\leq s-1$, there is $z\in [n]\setminus Y'$ with $S_z(Y')> a|Y'|-2$.  Apply this fact with $Y'=Y$ to choose $z_1\in [n]\setminus Y$ such that $S_{z_1}(Y)>a|Y|-2$.  Then inductively define a sequence $z_2,\ldots, z_{t-1}$ so that for each $1\leq i\leq t-2$, $S_{z_{i+1}}(Y\cup \{z_1,\ldots, z_i\})\geq a(s-t+1+i)-1$ (to define $z_{i+1}$, apply the fact with $Y'=Y\cup \{z_1,\ldots, z_i\}$).  Then $|Y\cup\{z_1,\ldots, z_{t-1}\}|=s$ and 
\begin{align*}
S(Y\cup \{z_1,\ldots, z_{t-1}\})&\geq S(Y)+S_{z_1}(Y)+S_{z_2}(Y\cup \{z_1\})+\ldots + S_{z_{t-1}}(Y\cup \{z_1,\ldots, z_{t-2}\})\\
&\geq a{s-t+1\choose 2}+a(s-t+1)-1+\ldots +a(s-1)-1\\
&=a{s\choose 2}-(t-1)>a{s\choose 2}-t,
\end{align*}
contradicting that $G\in F(n,s,q)$.  Therefore there is $Y\subseteq Y'\subseteq [n]$ such that $s-t+1\leq |Y'|\leq s-1$ and for all $z\in [n]\setminus Y'$,  $S_z(Y')\leq a|Y'|-2$.  By Lemma \ref{arithgeom}, this implies $P_z(Y')\leq a^{|Y'|-2}(a-1)^2$.
\end{proof}

\begin{lemma}\label{newlem}
Suppose $s,q,a,t$ are integers satisfying the hypotheses of case (b) of Theorem \ref{caseii}.  Then there are constants $C>1$ and $0<\alpha<1$ such that for all $n\geq 1$ the following holds.  Suppose $G\in F(n,s,q)$ and $k(G)$ is the maximal number of pairwise disjoint elements of $\{Y\in {[n]\choose s-t+1}: S(G[Y])\geq a{s-t+1\choose 2}\}$.  Then 
\begin{align}\label{caseiilem}
P(G)\leq C^{k(G)}\alpha^{k(G)n}\expi(n,s,q).
\end{align}
\end{lemma}
\begin{proof}
Set $\alpha=(\frac{a-1}{a})^{\frac{1}{2t(s-t)}}$.  Choose $C\geq q^{s-1\choose 2}$ sufficiently large so that $\expi(n,s,q)\leq C\alpha^{n^2}$ holds for all $1\leq n\leq s^3$.  We proceed by induction on $n$.  If $1\leq n\leq s^3$ and $G\in F(n,s,q)$, then (\ref{caseiilem}) is clearly true of $k(G)=0$.  If $k(G)\geq 1$, then by choice of $C$ and since $k(G)\leq n$ and $\alpha<1$,
$$
P(G)\leq \expi(n,s,q)\leq C\alpha^{n^2}\leq C\alpha^{k(G)n}\leq C^{k(G)}\alpha^{k(G)n}\expi(n,s,q).
$$
Now let $n>s^3$ and suppose by induction (\ref{caseiilem}) holds for all $G'\in F(n',s,q)$ where $1\leq n'<n$.  If $G\in F(n,s,q)$, then (\ref{caseiilem}) is clearly true if $k(G)=0$.  If $k(G)>0$, let $Y_1,\ldots, Y_k$ be a maximal set of pairwise disjoint elements in $\{Y\in {[n]\choose s-t+1}: S(G[Y])\geq a{s-t+1\choose 2}\}$.  Apply Lemma \ref{stabilityiiilemma1} to find $Y'$ such that $Y_1\subseteq Y'\subseteq [n]$, $s-t+1\leq |Y'|\leq s-1$, and for all $z\in [n]\setminus Y'$, $P_z(Y')\leq a^{|Y'|-2}(a-1)^2$.   Let $|Y'|=y$.  Then note
\begin{align}\label{ppp}
P(Y',[n]\setminus Y')=\prod_{z\in [n]\setminus Y'}P_z(Y')\leq \Big(a^{y-2}(a-1)^2\Big)^{n-y}.
\end{align}
Observe that $G[[n]\setminus Y']$ is isomorphic to some $H\in F(n-y,s,q)$.  Since $Y'$ can intersect at most $t-2$ other $Y_i$, and since $Y_1,\ldots, Y_k$ was maximal, we must have $k(H)+1\leq k(G)\leq k(H)+t-1$.  By our induction hypothesis, 
\begin{align}\label{p}
P([n]\setminus Y')=P(H)\leq C^{k(H)}\alpha^{k(H)(n-y)}\expi(n-y,s,q).
\end{align}
Since $\mu(G)\leq q$ and $y\leq s-1$, and by our choice of $C$, $P(Y')\leq q^{y\choose 2}\leq C$.  Combining this with (\ref{ppp}), (\ref{p}) and the fact that $\mu(H)\leq \mu(G)$ we obtain that 
\begin{align*}
P(G)=P([n]\setminus Y')P(Y',[n]\setminus Y')P(Y')&\leq  C^{k(H)}\alpha^{k(H)(n-y)}\expi(n-y,s,q)\Big(a^{y-2}(a-1)^2\Big)^{n-y}C\\
&= C^{k(H)+1}\alpha^{k(H)(n-y)}\expi(n-y,s,q)\Big(a^{y-2}(a-1)^2\Big)^{n-y}.
\end{align*}
Plugging in the upper bound for $\expi(n-y,s,q)$ from Claim \ref{gwthcaseii} yields that $P(G)$ is at most
\begin{align}
C^{k(H)+1}\alpha^{k(H)(n-y)}\expi(n,s,q)(a-1)^{-{y\choose 2}}\Big(\frac{a-1}{a}\Big)^{\frac{n-y}{s-t}} \leq C^{k(H)+1}\alpha^{k(H)(n-y)+2t(n-y)}\expi(n,s,q),\label{l}
\end{align}
where the last inequality is because $(a-1)^{-{y\choose 2}}<1$ and by definition of $\alpha$, $(\frac{a-1}{a})^{1/(s-t)}= \alpha^{2t}$.  We claim that the following holds.
\begin{align}\label{k}
k(H)(n-y)+2t(n-y)\geq (k(H)+t-1)n.
\end{align}
Rearranging this, we see (\ref{k}) is equivalent to $yk(H)\leq tn+n-2ty$.  Since $2\leq t\leq s/2$ and $y\leq s-1$, $tn+n-2ty\geq 3n-s(s-1)$, so it suffices to show $yk(H)\leq 3n-s(s-1)$.  By definition, $k(H)\leq \frac{n-y}{s-t+1}$ so $yk(H)\leq \frac{y(n-y)}{s-t+1}$.  Combining this with the facts that $s-t+1\leq y\leq s-1$ and $s/2< s-t+1$ yields 
\begin{align*}
yk(H)\leq \frac{(s-1)(n-(s-t+1))}{s-t+1}=n\Big(\frac{s-1}{s-t+1}\Big)-s+1< 2n\Big(\frac{s-1}{s}\Big)-s+1.
\end{align*}
Thus it suffices to check $2n(\frac{s-1}{s})-s+1\leq 3n-s(s-1)$.  This is equivalent to $(s-1)^2\leq n(\frac{s+2}{s})$, which holds because $n\geq s^3$.  This finishes the verification of (\ref{k}). Combining (\ref{l}), (\ref{k}), and the fact that $k(H)+1\leq k(G)\leq k(H)+t-1$ yields
\begin{align*}
P(G)\leq  C^{k(H)+1}\alpha^{(k(H)+t-1)n}\expi(n,s,q)\leq C^{k(G)}\alpha^{k(G)n}\expi(n,s,q).
\end{align*}
\end{proof}

\noindent{\bf Proof of Theorem \ref{caseii}(b) (Extremal).}
Set $s'=s-t+1$ and $q'=a{s'\choose 2}-1$.  Fix $n\geq s$.  By Lemma \ref{factii} and definition of $s'$, $\mathbb{T}_{s-t,a}(n)= \mathbb{T}_{s'-1,a}(n)\subseteq\calP(n,s,q)$ and 
$$
\expi(n,s,q)=\expi(n,s',q')=(a-1)^{n\choose 2}(\frac{a}{a-1})^{t_{s'-1}(n)},
$$
where the last equality is by Theorem \ref{caseii}(a) (Extremal) applied to $s'$ and $q'$.  By definition, we have $\expi(s,q)=(a-1)(\frac{a}{a-1})^{1-\frac{1}{s'-1}}$.  We have left to show that $\calP(n,s,q)\subseteq \mathbb{T}_{s'-1,a}(n)$ holds for large $n$.  Assume $n$ is sufficiently large and $C$ and $\alpha$ are as in Lemma \ref{newlem}.  Note $\expi(n,s,q)=\expi(n,s',q')$ implies $\calP(n,s,q)\cap F(n,s',q')\subseteq \calP(n,s',q')=\mathbb{T}_{s'-1,a}(n)$, where the equality is by Theorem \ref{caseii} (a) (Extremal).  So it suffices to show $\calP(n,s,q)\subseteq F(n,s',q')$.  Suppose towards a contradiction that there exists $G=([n],w)\in \calP(n,s,q)\setminus F(n,s',q')$.  Then in the notation of Lemma \ref{newlem}, $k(G)\geq 1$.  Combining this with Lemma \ref{newlem}, we have
$$
P(G)\leq C^{k(G)}\alpha^{k(G)n}\expi(n,s,q)=\Big(C\alpha^n\Big)^{k(G)}\expi(n,s,q)<\expi(n,s,q),
$$
where the last inequality is because $n$ is large, $\alpha<1$, and $k(G)\geq 1$.  But now $P(G)<\expi(n,s,q)$ contradicts that $G\in \calP(n,s,q)$. 
\qed

\section{Stability}\label{stabilitycasei-iii}

In this section we prove the product-stability results for Theorems~\ref{casei} and \ref{caseii}(a).  We will use the fact that for any $(s,q)$-graph $G$, $\mu(G)\leq q$.  If $G=(V,w)$ and $a\in \mathbb{N}$, let $E_a(G)=\{xy\in {V\choose 2}: w(xy)=a\}$ and $e_a(G)=|E_a(G)|$.  In the following notation, $p$ stands for ``plus'' and $m$ stands for ``minus.''
$$
p_a(G)=|\{xy\in {V\choose 2}: w(xy)>a\}|\qquad \hbox{ and } \qquad m_a(G)=|\{xy\in {V\choose 2}: w(xy)<a\}|.
$$

\begin{lemma}\label{generallem}
Let $s\geq 2$, $q\geq {s\choose 2}$ and $a>0$.  Suppose there exist $0<\alpha<1$ and $C>1$ such that for all $n\geq s$, every $G\in F(n,s,q)$ satisfies 
$$
P(G)\leq \expi(n,s,q)q^{Cn}\alpha^{p_a(G)}.
$$
Then for all $\delta>0$ there are $\epsilon, M>0$ such that for all $n>M$ the following holds.  If $G\in F(n,s,q)$ and $P(G)\geq \expi(n,s,q)^{1-\epsilon}$ then $p_a(G)\leq \delta n^2$.
\end{lemma}

\begin{proof}
Fix $\delta>0$. Choose $\epsilon>0$ so that $\frac{2\epsilon \log q}{\log(1/\alpha)}= \delta$.  Choose $M\geq s$ sufficiently large so that $n\geq M$ implies $(\epsilon n^2+Cn)\log q\leq 2\epsilon \log q n^2$.  Let $n>M$ and $G\in F(n,s,q)$ be such that $P(G)\geq \expi(n,s,q)^{1-\epsilon}$.  Our assumptions imply 
$$
\expi(n,s,q)^{1-\epsilon} \leq P(G)\leq \expi(n,s,q)q^{Cn}\alpha^{p_a(G)}.
$$
Rearranging $\expi(n,s,q)^{1-\epsilon}\leq \expi(n,s,q)q^{Cn}\alpha^{p_a(G)}$ yields $\Big(\frac{1}{\alpha}\Big)^{p_a(G)}\leq \expi(n,s,q)^{\epsilon}q^{Cn}\leq q^{\epsilon n^2+Cn}$, where the second inequality is because $\expi(n,s,q)\leq q^{n^2}$.  Taking logs of both sides, we obtain
$$
p_a(G)\log (1/\alpha) \leq (\epsilon n^2+Cn)\log q\leq 2\epsilon n^2 \log q,
$$
where the second inequality is by assumption on $n$.  Dividing both sides by $\log (1/\alpha)$ and applying the definition of $\epsilon$ yields $p_a(G)\leq \frac{2\epsilon n^2\log q}{\log(1/\alpha)}= \delta n^2$.
\end{proof}

\noindent We now prove the key lemma for this section.
\begin{lemma}\label{uniformstability}
Let $s,q,b,a$ be integers satisfying $s\geq 2$ and either 
\begin{enumerate}[(i)]
\item $a\geq 1$, $0\leq b\leq s-2$, and $q=a{s\choose 2}+b$ or 
\item $a\geq 2$, $b=0$, and $q=a{s\choose 2}-1$.
\end{enumerate}
Then there exist $0<\alpha<1$ and $C>1$ such that for all $n\geq s$ and all $G\in F(n,s,q)$, 
\begin{eqnarray}\label{*}
P(G)\leq \expi(n,s,q)q^{Cn}\alpha^{p_a(G)}.
\end{eqnarray}
\end{lemma}
\begin{proof}
We prove this by induction on $s\geq 2$, and for each fixed $s$, by induction on $n$.  Let $s\geq 2$ and $q,b,a$ be as in (i) or (ii) above.  Set
\[
\xi = \begin{cases} 0 &\text{ if case (i) holds}.\\
1&\text{ if case (ii) holds}.\end{cases}
\]
Suppose first $s=2$.  Set $\alpha=1/2$ and $C=2$.  Since $G$ is an $(n,2,a-\xi)$-graph, $p_a(G)=0$.  Therefore for all $n\geq 2$,
$$
P(G)\leq \expi(n,s,q)\leq \expi(n,s,q)q^{Cn}=\expi(n,s,q)q^{Cn}\alpha^{p_a(G)}.
$$
Assume now $s>2$.  Let $\mathcal{I}$ be the set of $(s',q',b')\in \mathbb{N}^3$ such that $2\leq s'<s$ and $s',q',b',a$ satisfy (i) or (ii).  Observe that $\mathcal{I}$ is finite.  Suppose by induction on $s$ that  $(s',q',b')\in \mathcal{I}$ implies there are $0<\alpha(s',q',b')<1$ and $C(s',q',b')>1$ such that for all $n'\geq s'$ and $G'\in F(n',s',q')$, $P(G)\leq \expi(n,s',q')q^{C(s',q',b')n}\alpha(s',q',b')^{p_a(G)}$.  Set
\begin{align*}
\alpha =\max \Big( \Big\{q^{-1},\Big(\frac{a^{s-2}(a-\xi)-1}{a^{s-2}(a-\xi)}\Big)^{\frac{1}{s-2}}, \Big(\frac{a-1}{a}\Big)^{\frac{1}{s-2}}\Big\}\cup \Big\{\alpha(s',q',b'):  (s',q',b')\in \mathcal{I}\Big\}\Big).
\end{align*}
Observe $0<\alpha<1$.  Choose $C \geq {s-1\choose 2}$ sufficiently large so that for all $n\leq s$
\begin{align}\label{ineq**}
q^{n\choose 2}\leq q^{C n}(a-\xi)^{{n\choose 2}}\Big(\frac{a}{a-\xi}\Big)^{t_{s-1}(n)}\alpha^{n\choose 2},
\end{align}
and so that for all $(s',q',b')\in \mathcal{I}$, $C(s',q',b')\leq C/2$ and $(\frac{a+1}{a})^{(s-3)/(s-2)}\leq q^{C/2}$.  Given $G\in F(n,s,q)$, set  
\begin{align*}
\Theta(G)=\Big\{Y\subseteq {[n]\choose s-1}: S(Y)\geq a{s-1\choose 2}+(1-\xi)b\Big\},
\end{align*} 
and let $A(n,s,q)=\{G\in F(n,s,q): \Theta(G)\neq \emptyset\}$. We show the following holds for all $n\geq 1$ and $G\in F(n,s,q)$ by induction on $n$. 
\begin{eqnarray}\label{ind*}
P(G)\leq q^{Cn}(a-\xi)^{{n\choose 2}}\Big(\frac{a}{a-\xi}\Big)^{t_{s-1}(n)}\alpha^{p_a(G)}.
\end{eqnarray}
This will finish the proof since $(a-\xi)^{n\choose 2}(\frac{a}{a-\xi})^{t_{s-1}(n)}\leq \expi(n,s,q)$ (by Theorem \ref{casei} (Extremal) for case (i) and Theorem \ref{caseii}(a) (Extremal) for case (ii)). If $n\leq s$ and $G\in F(n,s,q)$, then (\ref{ind*}) holds because of (\ref{ineq**}) and the fact that $P(G)\leq q^{n\choose 2}$.  So assume $n>s$, and suppose by induction that (\ref{ind*}) holds for all $s\leq n'<n$ and $G'\in F(n',s,q)$.  Let $G=([n],w)\in F(n,s,q)$.  Suppose first that $G\in A(n,s,q)$.  Choose $Y\in \Theta(G)$ and set $R=[n]\setminus Y$.  Given $z\in R$, note that
$$
a{s-1\choose 2}+(1-\xi)b+S_z(Y)\leq S(Y)+S_z(Y)=S(Y\cup\{z\})\leq a{s\choose 2}+(1-\xi)b-\xi,
$$
and therefore $S_z(Y)\leq a(s-1)-\xi$.  Then for all $z\in R$, Lemma \ref{arithgeom} implies $P_z(Y)\leq a^{s-2}(a-\xi)$, with equality only if $\{w(yz): y\in Y\}$ consists of $s-1-\xi$ elements equal to $a$ and $\xi$ elements equal to $a-1$.  Let $R_1=\{z\in R: \exists y\in Y, w(zy)>a\}$ and $R_2=R\setminus R_1$.  Then $z\in R_1$ implies $P_z(Y)<a^{s-2}(a-\xi)$, so $P_z(Y)\leq a^{s-2}(a-\xi)-1$.  Let $k=|R_1|$.  Observe that $G[R]$ is isomorphic to an element of $F(n', s,q)$, where $n'=n-|R|\geq 1$.  By induction (on $n$) and these observations we have that the following holds, where $p_a(R)=p_a({G[R]})$.
\begin{align*}
P(G)&=P(R)P(Y)\prod_{z\in R_1}P_z(Y)\prod_{z\in R_2}P_z(Y)\\
&\leq q^{C(n-s+1)}(a-\xi)^{{n-s+1\choose 2}}\Big(\frac{a}{a-\xi}\Big)^{t_{s-1}(n-s+1)}\alpha^{p_a(R)}q^{s-1\choose 2}\Big(a^{s-2}(a-\xi)-1\Big)^k\Big(a^{s-2}(a-\xi)\Big)^{n-s+1-k}\\
&\leq q^{C(n-s+2)}(a-\xi)^{{n-s+1\choose 2}}\Big(\frac{a}{a-\xi}\Big)^{t_{s-1}(n-s+1)}\alpha^{p_a(R)}\Big(a^{s-2}(a-\xi)-1\Big)^k\Big(a^{s-2}(a-\xi)\Big)^{n-s+1-k},
\end{align*}
where the second inequality is because ${s-1\choose 2}\leq C$.  Since $\alpha \geq \Big(\frac{a^{s-2}(a-\xi)-1}{a^{s-2}(a-\xi)}\Big)^{1/(s-2)}$,  this is at most
\begin{align}\label{ineq9}
q^{C(n-s+2)}(a-\xi)^{{n-s+1\choose 2}}\Big(\frac{a}{a-\xi}\Big)^{t_{s-1}(n-s+1)}\alpha^{p_a(R)+k(s-1)}\Big(a^{s-2}(a-\xi)\Big)^{n-s+1}.
\end{align}
Because $C(n-s+2)\leq Cn-{s-1\choose 2}$ and $q^{-1}\leq \alpha$,  we have $q^{C(n-s+2)}\leq q^{Cn}\alpha^{s-1\choose 2}$.  Combining this with the fact that $p_a(G)\leq p_a(R)+k(s-1)+{s-1\choose 2}$ implies that (\ref{ineq9}) is at most
\begin{align*}
&q^{Cn}(a-\xi)^{{n-s+1\choose 2}}\Big(\frac{a}{a-\xi}\Big)^{t_{s-1}(n-s+1)}\alpha^{p_a(R)+k(s-1)+{s-1\choose 2}}\Big(a^{s-2}(a-\xi)\Big)^{n-s+1}\\
= &q^{Cn}(a-\xi)^{{n-s+1\choose 2}+(s-1)(n-s+1)}\Big(\frac{a}{a-\xi}\Big)^{t_{s-1}(n-s+1)+(s-2)(n-s+1)}\alpha^{p_a(R)+k(s-1)+{s-1\choose 2}}\\
\leq &q^{Cn}(a-\xi)^{{n\choose 2}}\Big(\frac{a}{a-\xi}\Big)^{t_{s-1}(n)}\alpha^{p_a(G)}.
\end{align*}
We now have that $P(G)\leq q^{Cn}(a-\xi)^{{n\choose 2}}(\frac{a}{a-\xi})^{t_{s-1}(n)}\alpha^{p_a(G)}$, as desired.  Assume now $G\notin A(n,s,q)$.  Then for all $Y\in {[n]\choose s-1}$, $S(Y)\leq a{s-1\choose 2}+(1-\xi)b-1$.  Thus $G$ is an $(n,s',q')$-graph where $s'=s-1$ and $q'=a{s-1\choose 2}+(1-\xi)b-1$. Suppose $a=1$, $\xi=0$, and $b=0$.  Then $q'={s'\choose 2}-1$ and any $(n,s',q')$-graph must contain an edge of multiplicity $0$.  This implies $P(G)=0$ and (\ref{ind*}) holds.  We have the following three cases remaining, where $b'=\max\{b-1, 0\}$.
\begin{enumerate}
\item $\xi=0$, $b=0$, and $a\geq 2$.  In this case $q'=a{s'\choose 2}-1$ and $b'=0$.
\item $\xi=1$, $b=0$, and $a\geq 2$.  In this case $q'=a{s'\choose 2}-1$ and $b'=0$.
\item $\xi=0$, $1\leq b\leq s-2$, and $a\geq 1$.  In this case $q'=a{s'\choose 2}+b'$ and $0\leq b'\leq s'-2$.
\end{enumerate}
It is clear that in all three of these cases, $(s',q',b')\in \mathcal{I}$, so by our induction hypothesis (on $s$), there are $\alpha'=\alpha(s',q',b')\leq \alpha$ and $C'=C(s',q',b)$ such that
\begin{align}\label{indagain}
P(G)\leq \expi(n,s',q')(q')^{C'n}(\alpha')^{p_a(G)}\leq \expi(n,s',q')q^{C'n}\alpha^{p_a(G)},
\end{align}
where the inequality is because $q'\leq q$ and $\alpha'\leq \alpha$. By Theorem \ref{caseii}(a) (Extremal) if cases 1 or 2 hold, and by Theorem \ref{casei} (Extremal) if case 3 holds, we have the following.
\begin{align*}
\expi(n,s',q')\leq (a-\xi)^{n\choose 2}\Big(\frac{a}{a-\xi}\Big)^{t_{s'-1}(n)}\Big(\frac{a+1}{a}\Big)^{\lfloor \frac{b'}{b'+1}n\rfloor}\leq (a-\xi)^{n\choose 2}\Big(\frac{a}{a-\xi}\Big)^{t_{s-1}(n)}\Big(\frac{a+1}{a}\Big)^{\frac{s-3}{s-2}n},
\end{align*}
where the last inequality is because $t_{s'-1}(n)\leq t_{s-1}(n)$ and $\lfloor \frac{b'}{b'+1}n\rfloor \leq \frac{b'}{b'+1}n\leq \frac{s-3}{s-2}n$.  By choice of $C$, $(\frac{a+1}{a})^{\frac{s-3}{s-2}n}\leq q^{Cn/2}$.  Thus $\expi(n,s',q')\leq (a-\xi)^{n\choose 2}(\frac{a}{a-\xi})^{t_{s-1}(n)}q^{Cn/2}$.  Combining this with (\ref{indagain}) implies 
$$
P(G)\leq (a-\xi)^{n\choose 2}\Big(\frac{a}{a-\xi}\Big)^{t_{s-1}(n)}q^{Cn/2}q^{C'n}\alpha^{p_a(G)}\leq (a-\xi)^{n\choose 2}\Big(\frac{a}{a-\xi}\Big)^{t_{s-1}(n)}q^{Cn}\alpha^{p_a(G)},
$$
where the last inequality is because $C'\leq C/2$.  Thus (\ref{ind*}) holds.
\end{proof}

\noindent {\bf Proof of Theorem \ref{casei} (Stability).}
Let $s\geq 2$, $a\geq 1$, and $q=a{s\choose 2}+b$ for some $0\leq b\leq s-2$.  Fix $\delta>0$.  Given $G\in F(n,s,q)$, let $p_G=p_a(G)$ and $m_G=m_a(G)$.  Note that if $G\in F(n,s,q)$, then $|\Delta(G,\mathbb{U}_a(n))|=m_G+p_G$. Suppose first $a=1$, so $m_G=0$.  Combining Lemma \ref{uniformstability} with Lemma \ref{generallem} implies there are $\epsilon_1$ and $M_1$ such that if $n>M_1$ and $G\in F(n,s,q)$ satisfies $P(G)\geq \expi(n,s,q)^{1-\epsilon_1}$, then $|\Delta(G,\mathbb{U}_a(n))|=p_G\leq \delta n^2$.  Assume now $a>1$.  Combining Lemma \ref{uniformstability} with Lemma \ref{generallem} implies there are $\epsilon_1$ and $M_1$ such that if $n>M_1$ and $G\in F(n,s,q)$ satisfies $P(G)\geq \expi(n,s,q)^{1-\epsilon_1}$, then $p_G\leq \delta'n^2$, where
$$
\delta'= \min \Big\{ \frac{\delta }{2}, \frac{\delta \log (a/(a-1))}{4\log q}\Big\}.
$$
Set $\epsilon = \min\{\epsilon_1, \frac{\delta \log (a/(a-1))}{4\log q}\}$.  Suppose $n>M_1$ and $G\in F(n,s,q)$ satisfies $P(G)\geq \expi(n,s,q)^{1-\epsilon}$.  Our assumptions imply $p_G\leq \delta'n^2\leq \delta n^2/2$.  Observe that by definition of $p_G$ and $m_G$,
\begin{align}\label{labelcasei}
P(G)\leq a^{{n\choose 2}-m_G}(a-1)^{m_G}q^{p_G}=a^{n\choose 2}\Big(\frac{a-1}{a}\Big)^{m_G}q^{p_G}.
\end{align}
By Theorem \ref{casei}(a)(Extremal), $\expi(n,s,q)\geq a^{n\choose 2}$.  Therefore $P(G)\geq \expi(n,s,q)^{1-\epsilon}\geq a^{{n\choose 2}(1-\epsilon)}$. Combining this with (\ref{labelcasei}) yields 
$$
a^{{n\choose 2}(1-\epsilon)}\leq a^{n\choose 2}\Big(\frac{a-1}{a}\Big)^{m_G}q^{p_G}.
$$
Rearranging this, we obtain
$$
\Big(\frac{a}{a-1}\Big)^{m_G}\leq a^{\epsilon{n\choose 2}}q^{p_G}\leq q^{\epsilon {n\choose 2}+p_G}\leq q^{\epsilon n^2+p_G}.
$$
Taking logs, dividing by $\log(a/(a-1))$, and applying our assumptions on $p_G$ and $\epsilon$ yields 
$$
m_G\leq \frac{\epsilon n^2 \log q}{\log (a/(a-1))} + \frac{p_G\log q}{\log (a/(a-1)}\leq \frac{\delta n^2}{4}+\frac{\delta n^2}{4}=\frac{\delta n^2}{2}.
$$
Combining this with the fact that $p_G\leq \frac{\delta n^2}{2}$ we have that $|\Delta(G,\mathbb{U}_a(n))|\leq \delta n^2$.
\qed

\vspace{4mm}

The following classical result gives structural information about $n$-vertex $K_s$-free graphs with close to $t_{s-1}(n)$ edges.

\begin{theorem}[{\bf Erd\H{o}s-Simonovits \cite{erdosstability, simonovitsstability}}]\label{graphstability}
For all $\delta>0$ and $s\geq 2$, there is an $\epsilon>0$ such that every $K_{s}$-free graph with $n$ vertices and $t_{s-1}(n)-\epsilon n^2$ edges can be transformed into $T_{s-1}(n)$ by adding and removing at most $\delta n^2$ edges. 
\end{theorem}

\noindent {\bf Proof of Theorem \ref{caseii}(a) (Stability).}
Let $s\geq 2$, $a\geq 2$, and $q=a{s\choose 2}-1$.  Fix $\delta>0$.  Given $G\in F(n,s,q)$, let $p_G=p_a(G)$, $m_G=m_{a-1}(G)$.  Choose $M_0$ and $\mu$ such that $\mu<\delta/2$ and so that Theorem \ref{graphstability} implies that any $K_s$-free graph with $n\geq M_0$ vertices and at least $(1-\mu)t_{s-1}(n)$ edges can be made into $T_{s-1}(n)$ by adding or removing at most $\frac{\delta n^2}{3}$ edges.  Set 
\[ A=\begin{cases}
2 & \text{ if }a=2\\
\frac{a-1}{a-2}& \text{ if }a>2\end{cases}\]
Combining Lemma \ref{uniformstability} with Lemma \ref{generallem} implies there are $\epsilon_1, M_1$ so that if $n>M_1$ and $G\in F(n,s,q)$ satisfies $P(G)\geq \expi(n,s,q)^{1-\epsilon_1}$, then $p_G\leq \delta' n^2$, where 
\begin{align}\label{caseiistab*}
\delta'=\min \Big\{\frac{\delta }{3}, \frac{\mu \log (a/(a-1))}{2\log q}, \frac{\delta \log A}{6\log q}\Big\}.
\end{align}
Let  
$$
\epsilon=\min \Big\{\epsilon_1, \frac{\delta \log A}{6\log q}, \frac{\mu \log(a/(a-1))}{2\log q}\Big\} \qquad \hbox{ and }\qquad M=\max\{M_0,M_1\}.
$$
Suppose now that $n>M$ and $G\in F(n,s,q)$ satisfies $P(G)\geq \expi(n,s,q)^{1-\epsilon}$.  By assumption, $p_G\leq \delta'n^2\leq \frac{\delta n^2}{3}$.  We now bound $m_G$.  Note that if $a=2$ and $P(G)\neq 0$, then $m_G=0$.  If $a>2$, observe that by definition of $p_G$ and $m_G$,
\begin{align}\label{ineq*}
P(G)\leq q^{p_G}(a-2)^{m_G}a^{e_a(G)}(a-1)^{e_{a-1}(G)} \leq q^{p_G}\Big(\frac{a-2}{a-1}\Big)^{m_G}a^{e_a(G)}(a-1)^{{n\choose 2}-e_a(G)},
\end{align}
where the last inequality is because $e_{a-1}(G)+m_G\leq {n\choose 2}-e_a(G)$.  Note that Tur\'{a}n's theorem and the fact that $G$ is an $(n, s,q)$-graph implies that $e_a(G)\leq t_{s-1}(n)$, so 
$$
a^{e_a(G)}(a-1)^{{n\choose 2}-e_a(G)}\leq a^{t_{s-1}(n)}(a-1)^{{n\choose 2}-t_{s-1}(n)}=\expi(n,s,q),
$$
where the last equality is from Theorem \ref{caseii}(a) (Extremal).  Combining this with (\ref{ineq*}) yields 
$$
\expi(n,s,q)^{1-\epsilon}\leq P(G)\leq q^{p_G}\Big(\frac{a-2}{a-1}\Big)^{m_G}\expi(n,s,q).
$$
Rearranging $\expi(n,s,q)^{1-\epsilon}\leq q^{p_G}(\frac{a-2}{a-1})^{m_G}\expi(n,s,q)$ and using that $\expi(n,s,q)\leq q^{n^2}$, we obtain 
$$
A^{m_G}=\Big(\frac{a-1}{a-2}\Big)^{m_G}\leq q^{p_G}\expi(n,s,q)^{\epsilon}\leq q^{p_G+\epsilon n^2}.
$$
Taking logs, dividing by $\log A$, and applying our assumptions on $p_G$ and $\epsilon$ we obtain $m_G<\delta n^2/3$.  Using (\ref{ineq*}) and  $a^{t_{s-1}(n)}(a-1)^{{n\choose 2}-t_{s-1}(n)}=\expi(n,s,q)$, we have
$$
\expi(n,s,q)^{1-\epsilon}\leq P(G)\leq q^{p_G}a^{e_a(G)}(a-1)^{{n\choose 2}-e_a(G)}= q^{p_G}\expi(n,s,q)\Big(\frac{a}{a-1}\Big)^{e_a(G)-t_{s-1}(n)}. 
$$
Rearranging this we obtain 
$$
\Big(\frac{a}{a-1}\Big)^{t_{s-1}(n)-e_a(G)}\leq q^{p_G}\expi(n,s,q)^{\epsilon}\leq q^{p_G+\epsilon n^2}.
$$
Taking logs, dividing by $\log (a/(a-1))$, and using the assumptions on $p_G$ and $\epsilon$ we obtain that
$$
t_{s-1}(n) -e_a(G)\leq \frac{p_G\log q}{\log(a/(a-1))}+\frac{\epsilon n^2 \log q}{\log (a/(a-1))} \leq \frac{\mu n^2}{2}+\frac{\mu n^2}{2}=\mu n^2.
$$
Let $H$ be the graph with vertex set $[n]$ and edge set $E=E_a(G)$.  Then $H$ is $K_s$-free, and has $e_a(G)$ many edges.  Since $t_{s-1}(n)-e_a(G)\leq \mu n^2$, Theorem \ref{graphstability} implies that $H$ is $\frac{\delta}{3}$-close to some $H'= T_{s-1}(n)$.  Define $G'\in F(n,s,q)$ so that $E_a(G')=E(H')$ and $E_{a-1}(G')={n\choose 2}\setminus E_a(G')$.  Then $G'\in \mathbb{T}_{s-1, a}(n)$ and 
$$
\Delta(G,G')\subseteq (E_a(G)\Delta E_a(G'))\cup \bigcup_{i\notin \{a,a-1\}}E_i(G) = \Delta(H,H') \cup \bigcup_{i\notin \{a,a-1\}}E_i(G).
$$
This implies $|\Delta(G,G')|\leq |\Delta(H,H')|+p_G+m_G\leq \frac{\delta}{3}n^2+  \frac{\delta}{3}n^2+ \frac{\delta}{3}n^2 = \delta n^2$.
\qed

\vspace{5mm}

\subsection{Proof of Theorem \ref{caseii}(b) (Stability)}\label{caseiiisection}
In this subsection we prove Theorem \ref{caseii}(b) (Stability). We first prove two lemmas.  

\begin{lemma}\label{caseiiilemma}
Let $s\geq 4$, $a\geq 2$, and $q=a{s\choose 2}-t$ for some $2\leq t\leq \frac{s}{2}$.  For all $\lambda>0$ there are $M$ and $\epsilon>0$ such that the following holds.  Suppose $n>M$ and $G\in F(n,s,q)$ satisfies $P(G)>\expi(n,s,q)^{1-\epsilon}$.  Then $k(G)< \lambda n$, where $k(G)$ is as defined in Lemma \ref{newlem}.  
\end{lemma}
\begin{proof}
Fix $\lambda>0$.  Set $\eta =a^{\frac{s-t-1}{s-t}}(a-1)^{\frac{1}{s-t}}$ and choose $C$ and $\alpha$ as in Lemma \ref{newlem}.  Choose $\epsilon>0$ so that $\alpha^{\lambda /2}=\eta^{-\epsilon }$.  By Theorem \ref{caseii}(b) (Extremal), $\expi(n,s,q)=\eta^{{n\choose 2}+o(n^2)}$.   Assume $M$ sufficiently large so that for all $n\geq M$,  (\ref{caseii}) holds for all $G\in F(n,s,q)$, $\expi(n,s,q)< \eta^{n^2}$, $C^{\lambda n}\leq \eta^{\epsilon n^2}$,  and $C\alpha^n<1$.  Fix $n\geq M$ and suppose towards a contradiction that $G\in F(n,s,q)$ satisfies $P(G)>\expi(n,s,q)^{1-\epsilon}$ and $k(G)\geq \lambda n$.  By Lemma \ref{newlem} and the facts that $C\alpha^n<1$ and $k(G)\geq 1$, we obtain that
$$
P(G)\leq C^{k(G)}\alpha^{n k(G)}\expi(n,s,q)=(C\alpha^n)^{k(G)}\expi(n,s,q)\leq (C\alpha^n)^{\lambda n}\expi(n,s,q).
$$
By assumption on $n$ and definition of $\epsilon$, $(C\alpha^n)^{\epsilon n}=C^{\lambda n}\alpha^{\lambda n^2}=C^{\lambda n}\eta^{-2\epsilon n^2}\leq \eta^{-\epsilon n^2}$.  Thus 
$$
P(G)\leq \eta^{-\epsilon n^2}\expi(n,s,q)< \expi(n,s,q)^{1-\epsilon},
$$
where the last inequality is because by assumption, $\expi(n,s,q)< \eta^{n^2}$.  But this contradicts our assumption that $P(G)>\expi(n,s,q)^{1-\epsilon}$.
\end{proof}

Given a multigraph $G=(V,w)$, let $\calH(G,s,q)=\{Y\in {V\choose s}: S(Y)>q\}$.  Observe that $G$ is an $(s,q)$-graph if and only if $\calH(G,s,q)=\emptyset$. 

\begin{lemma}\label{triangleremoval}
Let $s,q, m\geq 2$ be integers.  For all $0<\delta<1$, there is $0<\lambda<1$ and $N$ such that $n>N$ implies the following.  If $G=([n],w)$ has $\mu(G)\leq m$ and $\calH(G,s,q)$ contains strictly less than $\lceil \lambda n\rceil$ pairwise disjoint elements, then $G$ is $\delta$-close to an element in $F(n,s,q)$.
\end{lemma}
\begin{proof}
Fix $0<\delta<1$.  Observe we can view any multigraph $G$ with $\mu(G)\leq m$ as an edge-colored graph with colors in $\{0,\ldots, m\}$. By Theorem \ref{triangleremoval1}, there is $\epsilon$ and $M$ such that if $n>M$ and $G=([n],w)$ has $\mu(G)\leq m$ and $\calH(G,s,q)\leq \epsilon {n\choose s}$, then $G$ is $\delta$-close to an element of $F(n,s,q)$.  Let $\lambda:=\epsilon/s$ and $N=\max \{M, \frac{s}{1-\lambda s}\}$.  We claim this $\lambda$ and $N$ satisfy the desired conclusions.  Suppose towards a contradiction that $n>M$ and $G=([n],w)$ has $\mu(G)\leq m$, $\calH(n,s,q)$ contains strictly less than $\lceil \lambda n \rceil$ pairwise disjoint elements, but $G$ is $\delta$-far from every element in $F(n,s,q)$.  Then $\calH(G,s,q)> \epsilon {n\choose s}$ by choice of $M$ and $\lambda$.  By our choice of $N$, $ \lceil \lambda n \rceil s \leq (\lambda n+1)s\leq n$.  Then Proposition 11.6 in \cite{handbookcombo2} and our assumptions imply $|\calH(G,s,q)|\leq (\lceil\lambda n \rceil-1){n-1\choose s-1}$.  But now
$$
|\calH(G,s,q)|\leq (\lceil \lambda n \rceil -1){n-1\choose s-1}< \lambda n {n-1\choose s-1}=\Big(\frac{\epsilon n}{s}\Big)\Big(\frac{s}{n}\Big){n\choose s}=\epsilon {n\choose s},
$$
a contradiction.
\end{proof}

\noindent {\bf Proof of Theorem \ref{caseii}(b) (Stability).}
Let $s\geq 4$, $a\geq 2$, and $q=a{s\choose 2}-t$ for some $2\leq t\leq \frac{s}{2}$.  Fix $\delta>0$.  Let $s'=s-t+1$ and $q'=a{s'\choose 2}-1$.  Note Theorem \ref{caseii} (Extremal) implies that for sufficiently large $n$, $\calP(n,s,q)=\mathbb{T}_{s'-1,a}(n)$, $\expi(n,s',q')=\expi(n,s,q)$, and $\expi(s',q')=\expi(s,q)=\eta$, where $\eta=(a-1)(\frac{a}{a-1})^{(s'-2)/(s'-1)}$.   

Apply Theorem \ref{caseii} (a) (Stability) for $(s',q')$ to $\delta/2$ to obtain $\epsilon_0$.     By replacing $\epsilon_0$ if necessary, assume $\epsilon_0<4\delta/\log \eta$.  Set $\epsilon_1=\epsilon_0\log \eta/(8\log q)$ and note $\epsilon_1<\delta/2$.  Apply Lemma \ref{triangleremoval} to $\epsilon_1$ and $m=q$ to obtain $\lambda$ such that for large $n$ the following holds.  If $G=([n],w)$ has $\mu(G)\leq q$ and $\calH(G,s',q')$ contains strictly less than $\lceil \lambda n\rceil$ pairwise disjoint elements, then $G$ is $\epsilon_1$-close to an element in $F(n,s',q')$.  Finally, apply Lemma \ref{caseiiilemma} for $s, q, t$ to $\lambda$ to obtain $\epsilon_2>0$.

Choose $M$ sufficiently large for the desired applications of Theorems \ref{caseii}(a) (Stability) and \ref{caseii}(b) (Extremal) and Lemmas \ref{caseiiilemma} and \ref{triangleremoval}.  Set $\epsilon =\min \{\epsilon_2, \epsilon_0/2\}$.  Suppose $n>M$ and $G\in F(n,s,q)$ satisfies $P(G)\geq \expi(n,s,q)^{1-\epsilon}$.  Then Lemma \ref{caseiiilemma} and our choice of $\epsilon$ implies $k(G)< \lambda n$.  Observe that by the definitions of $s',q'$,
$$
\Big\{Y\in {[n]\choose s-t+1}: S(Y)\geq a{s-t+1\choose 2}\Big\}=\Big\{Y\in {[n]\choose s'}: S(Y)\geq q'+1\Big\}=\calH(G,s',q').
$$
Thus $k(G)<  \lambda n$ means $\calH(G,s',q')$ contains strictly less than $\lceil \lambda n\rceil$ pairwise disjoint elements.  Lemma \ref{triangleremoval} then implies $G$ is $\epsilon_1$-close to some $G'\in F(n,s',q')$.  Combining this with the definition of $\epsilon_1$ yields 
\begin{align}\label{caseiiiineq}
P(G')\geq P(G)q^{-|\Delta(G,G')|}\geq P(G)q^{-\epsilon_1 n^2}= P(G)\eta^{-\epsilon_0n^2/8}\geq \expi(n,s,q)^{1-\epsilon}\eta^{-(\epsilon_0/2){n\choose 2}}.
\end{align}
By Proposition \ref{density}, $\expi(n,s,q)\geq \expi(s,q)^{n\choose 2}= \eta^{n\choose 2}$.  Combining this with (\ref{caseiiiineq}) and the definition of $\epsilon$ yields 
\begin{align}\label{caseiii3}
P(G')\geq \expi(n,s,q)^{1-\epsilon}\eta^{-(\epsilon_0/2){n\choose 2}}\geq \expi(n,s,q)^{1-\epsilon -\epsilon_0/2}\geq \expi(n,s,q)^{1-\epsilon_0}.
\end{align}
Since $\expi(n,s,q)=\expi(n,s',q')$, (\ref{caseiii3}) implies $P(G')\geq \expi(n,s',q')^{1-\epsilon_0}$, so Theorem \ref{caseii}(a) (Stability) implies $G'$ is $\delta/2$-close to some $G''\in \mathbb{T}_{s'-1,a}(n)=\mathbb{T}_{s-t,a}(n)$.  Now we are done, since 
$$
|\Delta(G,G'')|\leq |\Delta(G,G')|+|\Delta(G',G'')|\leq \epsilon_1n^2 +\delta n^2/2 \leq \delta n^2.
$$
\qed

\section{Extremal Result for $(n,4,9)$-graphs}\label{extcaseivsection1}


In this section we prove Theorems \ref{caseiv1}.  We first prove one of the inequalities needed for Theorem \ref{caseiv1}.

\begin{lemma}\label{fn1}
For all $n\geq4$, $2^{\ex(n,\{C_3,C_4\})}\leq \expi(n,4,9)$.
\end{lemma}
\begin{proof}
Fix $G=([n],E)$ an extremal $\{C_3,C_4\}$-free graph, and let $G'=([n],w)$ where $w(xy)=2$ for all $xy\in E$ and $w(xy)=1$ for all $xy\in {n\choose 2}\setminus E$.  Suppose $X\in {[n]\choose 4}$.  Since $G$ is $\{C_3,C_4\}$-free, $|E\cap{X\choose 2}|\leq 3$.  Thus $\{w(xy): xy\in {X\choose 2}\}$ contains at most $3$ elements equal to $2$ and the rest equal to $1$, so $S(X)\leq 9$.  This shows $G'\in F(n,4,9)$.  Thus $2^{|E|}=2^{\ex(n,\{C_3,C_4\})}=P(G')\leq \expi(n,4,9)$.
\end{proof}

To prove the reverse inequality, our strategy will be to show that if $G\in F(n,4,9)$ has no edges of multiplicity larger than $2$, then $P(G)\leq 2^{\ex(n,\{C_3,C_4\})}$ (Theorem \ref{fn3}).  We will then show that all product extremal $(4,9)$-graphs have no edges of multiplicity larger than $2$ (Theorem \ref{reductionlemma}).  Theorem \ref{caseiv1} will then follow.  We begin with a few definitions and lemmas.

\begin{definition}
Suppose $n\geq 1$.  Set $F_{\leq 2}(n,4,9)=\{G\in F(n,4,9): \mu(G)\leq 2\}$ and 
$$
D(n)=F_{\leq 2}(n,4,9)\cap F(n,3,5).
$$
\end{definition}

\begin{lemma}\label{reduction}
For all $n\geq 4$, if $G=([n],w)\in D(n)$, then $P(G)\leq 2^{\ex(n,\{C_3,C_4\})}$. 
\end{lemma}
\begin{proof} If $P(G)=0$ we are done, so assume $P(G)>0$.  Let $H=([n],E)$ be the graph where $E=\{xy\in {[n]\choose 2}: w(xy)=2\}$.  Since $P(G)>0$ and $\mu(G)\leq 2$, $G$ contains all edges of multiplicity $1$ or $2$.  Consequently, $P(G)=2^{|E|}$.  Since $G\in F(n,3,5)$, $H$ is $C_3$-free and since $G\in F(n,4,9)$, $H$ is $C_4$-free, so $|E|\leq \ex(n, \{C_3,C_4\})$.  This shows $P(G)=2^{|E|}\leq 2^{\ex(n,\{C_3,C_4\})}$.
\end{proof}

\vspace{3mm}
The following lemma gives us useful information about elements of $F(n,4,9)\setminus F(n,3,5)$.

\begin{lemma}\label{fn2}
Suppose $n\geq 4$ and $G=([n],w)\in F(n,4,9)$ satisfies $P(G)>0$.  If there is $X\in {[n]\choose 3}$ such that $S(X)\geq 6$, then $P(X)\leq 2^3$ and $w(xy)=1$ for all $x\in X$ and $y\in [n]\setminus X$.  Consequently
\begin{align*}
P(G)=P(X)P([n]\setminus X)\leq 2^3P([n]\setminus X).
\end{align*}
\end{lemma}
\begin{proof}
Let $y\in [n]\setminus X$.  Since $P(G)>0$, every edge in $G$ has multiplicity at least $1$, so $S_y(X)\geq 3$.  Thus
$$
3+S(X)\leq S_y(X)+S(X)=S(X \cup \{y\})\leq 9,
$$
which implies $S(X)\leq 6$.  By Lemma \ref{arithgeom}, this implies $P(X)\leq 2^3$.  By assumption, $S(X)\geq 6$, so we have $6+S_y(X)\leq S(X)+S_y(X)=S(X \cup \{y\})\leq 9$, which implies $S_y(X)\leq 3$.  Since every edge in $G$ has multiplicity at least $1$ and $|X|=3$, we must have $w(yx)=1$ for all $x\in X$.  Therefore $P(G)=P([n]\setminus X)P(X)\leq P([n]\setminus X)2^3$.
\end{proof}

\begin{fact}\label{c4fact}
For all $n\geq 4$ and $1\leq i< n$, $\ex(n,\{C_3,C_4\})\geq \ex(n-i,\{C_3,C_4\})+i$.  
\end{fact}

\begin{proof}
Suppose $n\geq 4$ and $1\leq i< n$.  Fix $G=([n-i],E)$ an extremal $\{C_3,C_4\}$-free graph.  Let $G'=([n],E')$ where $E'=E\cup \{n1, (n-1)1,\ldots, (n-i+1)1\}$.  Then $G'$ is $\{C_3,C_4\}$-free graph because $G=G'[n-i]$ is $\{C_3,C_4\}$-free and because the elements of $[n]\setminus [n-i]$ all have degree $1$ in $G'$.  Therefore $\ex(n,\{C_3,C_4\})\geq \ex(n-i,\{C_3,C_4\})+|E'\setminus E|=\ex(n-i,\{C_3,C_4\})+i$.
\end{proof}

We now prove Theorem \ref{fn3}.  We will use that $\ex(4,\{C_3,C_4\})=3$, $\ex(5,\{C_3,C_4\})=5$, and $\ex(6,\{C_3,C_4\})=6$ (see \cite{GKHF}).

\begin{theorem}\label{fn3}
For all $n\geq 4$ and $G \in F_{\leq 2}(n,4,9)$, $P(G)\leq 2^{\ex(n,\{C_3,C_4\})}$.
\end{theorem}
\begin{proof}
We proceed by induction on $n$. Assume first $4\leq n\leq 6$ and $G\in F_{\leq 2}(n,4,9)$.  If $P(G)=0$ then we are done.  If $G\in D(n)$, then we are done by Lemma \ref{reduction}.  So assume $P(G)>0$ and $G\in F_{\leq 2}(n,4,9)\setminus D(n)$.  By definition of $D(n)$ this means $G\notin F(n,3,5)$, so there is $X\in {[n]\choose 3}$ such that $S(X)\geq 6$.  By Lemma \ref{fn2}, this implies $P(G)\leq P([n]\setminus X)2^3\leq 2^{{n-3\choose 2}+3}$, where the second inequality is because $\mu(G)\leq 2$.  The explicit values for $\ex(n,\{C_3,C_4\})$ show that for $n\in \{4,5,6\}$, $2^{{n-3\choose 2}+3}\leq 2^{\ex(n,\{C_3,C_4\})}$.  Consequently, $P(G)\leq 2^{{n-3\choose 2}+3}\leq 2^{\ex(n,\{C_3,C_4\})}$.

Suppose now $n\geq 7$ and assume by induction that for all $4\leq n'<n$ and $G'\in F_{\leq 2}(n',4,9)$, $P(G')\leq 2^{\ex(n',4,9)}$.  Fix $G\in F_{\leq 2}(n,4,9)$.  If $P(G)=0$ then we are done.  If $G\in D(n)$, then we are done by Lemma \ref{reduction}.  So assume $P(G)>0$ and $G\in F_{\leq 2}(n,4,9)\setminus D(n)$.  By definition of $D(n)$ this means $G\notin F(n,3,5)$, so there is $X\in {[n]\choose 3}$ such that $S(X)\geq 6$.  By Lemma \ref{fn2}, this implies $P(G)\leq P([n]\setminus X)2^3$. Clearly there is $H\in F_{\leq 2}(n-3,4,9)$ such that $G[[n]\setminus X]\cong H$.  By our induction hypothesis applied to $H$, $P([n]\setminus X)=P(H)\leq 2^{\ex(n-3,\{C_3,C_4\})}$.  Therefore 
$$
P(G)\leq P([n]\setminus X)2^3\leq 2^{\ex(n-3,\{C_3,C_4\})+3}\leq 2^{\ex(n,\{C_3,C_4\})},
$$
where the last inequality is by Fact \ref{c4fact} with $i=3$.
\end{proof}

We will use the following lemma to prove Theorem \ref{reductionlemma}.  Observe for all $n\geq 2$, $\expi(n,4,9)>0$ implies that for all $G\in \calP(n,4,9)$, every edge in $G$ has multiplicity at least $1$.  We will write $xyz$ to denote the three element set $\{x,y,z\}$.

\begin{lemma}\label{prelimlem}
Suppose $n\geq 4$ and $G=([n],w)\in \calP(n,4,9)$ satisfies $\mu(G)\geq 3$.  Then one of the following hold.
\begin{enumerate}[(i)]
\item There is $xyz\in {[n]\choose 3}$ such that $\mu(G[[n]\setminus xyz])\leq 2$ and $P(G)\leq 6\cdot P([n]\setminus xyz)$.
\item There is $xy\in {[n]\choose 2}$ such that $\mu(G[[n]\setminus xy])\leq 2$ and $P(G)\leq 3\cdot P([n]\setminus xy)$.
\end{enumerate}
\end{lemma}
\begin{proof}
Suppose $n\geq 4$ and $G=([n],w)\in \calP(n,4,9)$ is such that $\mu(G)\geq 3$.  Fix $xy\in {[n]\choose 2}$ such that $w(xy)=\mu(G)$.  We begin by proving some preliminaries about $G$ and $xy$.  We first show $w(xy)=3$.  By assumption, $w(xy)\geq 3$.  Suppose towards a contradiction $w(xy)\geq 4$.  Choose some $u\neq v \in [n]\setminus xy$.  Since every edge in $G$ has multiplicity at least $1$, $5+w(xy)\leq S(\{x,y,u,v\})\leq 9$.  This implies $w(xy)\leq 9-5=4$, and consequently $w(xy)=4$.  Combining this with the fact that every edge has multiplicity at least $1$, we have 
$$
9\leq 4+w(uv)+w(ux)+w(vx)+w(yu)+w(yv)=S(\{x,y,u,v\})\leq 9.
$$
Consequently, $w(uv)=w(ux)=w(vx)=w(yu)=w(yv)=1$.  Since this holds for all pairs $uv\in {[n]\choose 2}\setminus xy$, we have shown $P(G)=w(xy)=4$.  Because $n\geq 4$, Fact \ref{c4fact} implies 
$$
2^{\ex(n,\{C_3,C_4\})}\geq 2^{\ex(4,\{C_3,C_4\})}=2^3>4=P(G).
$$ 
Combining this with Lemma \ref{fn1} shows $P(G)<2^{\ex(n,\{C_3,C_4\})}\leq \expi(n,4,5)$, a contradiction. Thus $\mu(G)=w(xy)=3$.   We now show that for all $uv\in {[n]\choose 2}\setminus xy$, $w(uv)\leq 2$. Fix $uv\in {[n]\choose 2}\setminus xy$ and suppose towards a contradiction $w(uv)\geq 3$.  Choose some $X\in {[n]\choose 4}$ containing $\{x,y,u,v\}$.  Because every edge in $G$ has multiplicity at least $1$, we have that $S(X)\geq w(uv)+w(xy)+4\geq 10$, a contradiction.  Thus $w(uv)\leq 2$ for all $uv\in {[n]\choose 2}\setminus xy$.  We now show that for all $z\in [n]\setminus xy$, at most one of $w(xz)$ or $w(yz)$ is equal to $2$.  Suppose towards a contradiction there is $z\in [n]\setminus xy$ such that $w(zx)=w(zy)=2$.   Note $S(xyz)\geq 7$. So for each $z'\in [n]\setminus xyz$, $S_{z'}(xyz)\leq 9-S(xyz)=9-7=2$.  But since every edge has multiplicity at least $1$ this is impossible.  Thus for all $z\in [n]\setminus xy$, at most one of $w(xz)$ or $w(yz)$ is equal to $2$.

We now prove either (i) or (ii) holds.  Suppose there is $z\in [n]\setminus xy$ such that one of $w(zx)$ or $w(zy)$ is equal to $2$.  Then by what we have shown, $\{w(xy), w(zx),w(zy)\}=\{3,1,2\}$, and consequently $P(xyz)=6$.  By Lemma \ref{fn2}, since $S(xyz)\geq 6$, we have that
$$
P(G)=P(xyz)P([n]\setminus xyz)= 6\cdot P([n]\setminus xyz).
$$
By the preceding arguments, $\mu(G[[n]\setminus xyz]))\leq 2$.  Thus (i) holds. Suppose now that for all $z\in [n]\setminus xy$, $w(xz)=w(yz)=1$.  Then $P(G)=w(xy)P([n]\setminus xy)= 3\cdot P([n]\setminus xy)$.  By the preceding arguments, $\mu(G[[n]\setminus xy])\leq 2$. Thus (ii) holds.
\end{proof}

\begin{theorem}\label{reductionlemma}
For all $n\geq 4$, $\calP(n,4,9)\subseteq F_{\leq 2}(n,4,9)$.
\end{theorem}
\begin{proof}
Fix $n\geq 4$ and $G=([n],w)\in \calP(n,4,9)$.  Suppose towards a contradiction $G\notin F_{\leq 2}(n,4,9)$.  We show $P(G)<2^{\ex(n,\{C_3,C_4\})}$, contradicting that $G$ is product-extremal (since by Lemma \ref{fn1}, {$2^{\ex(n,\{C_3,C_4\})}\leq \expi(n,4,9)$}). 

Since $G\notin F_{\leq 2}(n,4,9)$, either (i) or (ii) of Lemma \ref{prelimlem} holds.  If (i) holds, choose $xyz\in {[n]\choose 3}$ with $\mu(G[[n]\setminus xyz])\leq 2$ and $P(G)\leq 6\cdot P([n]\setminus xyz)$.  Let $H\in F_{\leq 2}(n-2,4,9)$ be such that $G[[n]\setminus xy]\cong H$.  If $n\in \{4,5,6\}$, then $P(G)\leq 6\cdot P(H)\leq 6\cdot 2^{{n-3\choose 2}}<2^{\ex(n,\{C_3,C_4\})}$, where the second inequality is because $\mu(H)\leq 2$, and the strict inequality is from  the exact values for $\ex(n,\{C_3,C_4\})$ for $n\in \{4,5,6\}$.  If $n\geq 7$, then by Lemma \ref{fn3} and because $n-3\geq 4$, $P(H)\leq 2^{\ex(n-3,\{C_3,C_4\})}$.  Therefore,
\begin{align*}
P(G)\leq 6\cdot P(H)\leq 6\cdot 2^{\ex(n-3,\{C_3,C_4\})}<2^{\ex(n-3,\{C_3,C_4\})+3}\leq 2^{\ex(n,\{C_3,C_4\})},
\end{align*}
where the last inequality is by Fact \ref{c4fact}. If (ii) holds, choose $xy\in {[n]\choose 2}$ with $\mu(G[[n]\setminus xy]))\leq 2$ and $P(G)\leq 3 \cdot P([n]\setminus xy)$.  Let $H\in F_{\leq 2}(n-2,4,9)$ be such that $G[[n]\setminus xy]\cong H$.  If $n\in \{4,5\}$, then $P(G)\leq 3\cdot P(H)\leq 3\cdot 2^{n-2\choose 2}<2^{\ex(n,\{C_3,C_4\})}$, where the second inequality is because $\mu(H)\leq 2$, and the strict inequality is from the exact values for $\ex(n,\{C_3,C_4\})$ for $n\in \{4,5\}$. If $n\geq 6$, then $n-2\geq 4$ and Lemma \ref{fn3} imply $P(H)\leq 2^{\ex(n-2,\{C_3,C_4\})}$.  Therefore,
\begin{align*}
P(G)\leq 3\cdot P([n]\setminus xy)\leq 3 \cdot 2^{\ex(n-2,\{C_3,C_4\})}<2^{\ex(n-2,\{C_3,C_4\})+2}\leq 2^{\ex(n,\{C_3,C_4\})},
\end{align*}
where the last inequality is by Fact \ref{c4fact}.
\end{proof}

\noindent {\bf Proof of Theorem \ref{caseiv1}.}  Fix $n\geq 4$ and $G\in \calP(n,4,9)$.  By Theorem \ref{reductionlemma}, $G\in F_{\leq 2}(n,4,9)$.  By Theorem \ref{fn3}, this implies $P(G)\leq 2^{\ex(n,\{C_3,C_4\})}$.  By Lemma \ref{fn1}, $P(G)\geq 2^{\ex(n,\{C_3,C_4\})}$.  Consequently, $P(G)=2^{\ex(n,\{C_3,C_4\})}=\expi(n,4,9)$.
\qed

\bibliography{/Users/carolineterry/Desktop/refs.bib}
\bibliographystyle{amsplain}

\end{document}